\documentclass[11pt]{article}
\usepackage{}
\usepackage{amsfonts}
\usepackage{amssymb}
\usepackage{amsmath}
\usepackage{dutchcal}
\usepackage{mathrsfs}

\def\sqr#1#2{{\vcenter{\vbox{\hrule height.#2pt
              \hbox{\vrule width.#2pt height#1pt \kern#1pt \vrule
width.#2pt}
              \hrule height.#2pt}}}}
\def\signed #1{{\unskip\nobreak\hfil\penalty50
              \hskip2em\hbox{}\nobreak\hfil#1
              \parfillskip=0pt \finalhyphendemerits=0 \par}}
\def\endpf{\signed {$\sqr69$}}
\def\3n{\negthinspace \negthinspace \negthinspace }
\def\2n{\negthinspace \negthinspace }
\def\1n{\negthinspace }
\def\ns{\noalign{\smallskip} }
\def\ns{\noalign{\medskip} }

\def\no{\noindent}

\def\({\Big (}
\def\){\Big )}
\def\[{\Big[}
\def\]{\Big]}

\def\bel{\begin{equation}\label}
\def\ee{\end{equation}}
\def\bea{\begin{eqnarray}}
\def\eea{\end{eqnarray}}
\def\bt{\begin{theorem}}
\def\et{\end{theorem}}
\def\bc{\begin{corollary}}
\def\ec{\end{corollary}}
\def\bl{\begin{lemma}}
\def\el{\end{lemma}}
\def\bp{\begin{proposition}}
\def\ep{\end{proposition}}
\def\br{\begin{remark}}
\def\er{\end{remark}}
\def\ba{\begin{array}}
\def\ea{\end{array}}
\def\bd{\begin{definition}}
\def\ed{\end{definition}}

\newtheorem{lemma}{Lemma}[section]
\newtheorem{remark}{Remark}[section]

\newtheorem{theorem}{Theorem}[section]
\newtheorem{corollary}{Corollary}[section]

\newtheorem{definition}{Definition}[section]
\newtheorem{proposition}{Proposition}[section]

\makeatletter
   
   \@addtoreset{equation}{section}
\makeatother

\begin{document}

\title{\bf   Optimal control problems for quasi-linear parabolic
equations and their linear approximations\thanks{This work was partially supported by the National Key R\&D Program of China
under grants 2023YFA1009002 and 2024YFA1013400, the NSF of China under grant 12371444, the
New Cornerstone Investigator Program, and the Science Development Project of Sichuan University
under grant 2020SCUNL201.}}

\author{Xu Liu\thanks{School of Mathematics and Statistics, Northeast Normal
University, Changchun 130024, China.   E-mail:
liux216@nenu.edu.cn.} \quad Meisai Wang\thanks{School of Mathematics and Statistics, Northeast Normal
University, Changchun 130024, China.   E-mail: wangms901@nenu.edu.cn.
} \quad and \quad  Xu Zhang\thanks{School of Mathematics and New Cornerstone Science Laboratory,  Sichuan University,
Chengdu  610064, China. E-mail address: zhang\_xu@scu.edu.cn.}}

\date{}

\maketitle

\begin{abstract}
Motivated  by   thermal diffusion models,
a controlled quasi-linear parabolic equation
 can be formally replaced by its linear approximation,  when the initial datum,
 target and control actions
 are sufficiently small.
This paper makes this approximation precise at the level
of optimal control problems.  First, we provide
verification criteria which distinguish whether
an optimal control is an interior point or a boundary point of the admissible control set.
  We
then prove that the differences between the optimal controls
and the corresponding optimal
trajectories for the quasi-linear and linear parabolic equations are
higher-order infinitesimals
with respect to the small initial datum and target appearing in the cost functional.
 Consequently, for small data,
  the optimal control problem governed by the quasi-linear equation can
be approximated by the corresponding linear problem. In the interior-point case,
 the resulting
approximation order is sharp. \end{abstract}

\medskip

\no{\bf Key Words}.
Quasi-linear parabolic equation,  linear parabolic equation,
     optimal control,  optimal trajectory

\medskip

\no{\bf AMS subject classifications}.  35K59, 49J20, 49K20, 93C20.

\date{}
\maketitle

\section{Introduction}

In many engineering practices, nonlinear models that describe complex dynamic processes are often approximated by linear models.
When the approximation error is kept within an acceptable range,  the simplified linear models
can improve  computational efficiency, reduce costs and provide a more
tractable analytical framework. It is therefore natural to ask when a nonlinear model and its  control
problems can be reliably replaced by their linear counterparts.

\smallskip

The linear-quadratic (LQ)  problem stands as one of milestones in optimal control theory.
It
concerns linear dynamical systems with quadratic cost functionals,  and seeks an optimal balance
between tracking performance and control effort.
 The LQ problems  have broad applications in many fields such as aerospace engineering,
 chemical process control and environmental engineering. For instance, during spacecraft re-entry,
  intense aerodynamic friction produces high heat flux and a highly nonuniform
  temperature field on the surface. Distributed heating or cooling devices may then be regulated so
that the surface temperature remains within safe limits while the energy expenditure is minimized.
This type of design can be naturally formulated as an LQ problem. The LQ framework is also
attractive,  because it admits efficient numerical methods and, in many settings, feedback laws with
a transparent structure.

\smallskip

This paper is mainly concerned with the connection between  optimal control problems for quasi-linear parabolic equations and the corresponding LQ problem. Let  $T>0$,  $n\in\mathbb N$ and  $\Omega $ be a   bounded domain in
$\mathbb{R}^n$  with a smooth  boundary $\Gamma$.
Put $Q=\Omega\times(0,T)$ and $\Sigma=\Gamma\times(0,T)$.
Assume that $\omega$ is  a  nonempty open subset of $\Omega$ satisfying
  $\overline{\omega}\subseteq\Omega$, and
$\chi_{\omega}$ denotes the characteristic function on the set $\omega$.
Consider the following controlled quasi-linear parabolic equation:
\begin{eqnarray}\label{aeq11}
	\left\{
	\begin{array}{ll}
		y_t-\sum\limits_{i,j=1}^{n}(a^{i j}(y)y_{x_i})_{x_j}+f(y)=\chi_{\omega} u &\mbox{ in }Q, \\ \ns
		y=0 &\mbox{ on }\Sigma, \\ \ns
		y(x, 0)=y_0(x) &\mbox{ in }\Omega, \\ \ns
	\end{array}
	\right.
\end{eqnarray}
where $u$ is the control  variable and $y$ is the state variable,
$y_0$ is a given initial value,
and
$f(\cdot) \in C^2(\mathbb R)$ with $f(0)=0$. Further, $a^{i j}(\cdot)\in C^2(\mathbb R)$ satisfies
$a^{i j}=a^{j i}$ $(i, j=1, \cdots, n)$, and there exists a positive constant $\rho_0$, such that
$$
\sum\limits_{i,j=1}^n a^{i j}(s)\zeta_i\zeta_j\geq \rho_0 |\zeta|^2, \quad
\forall\ (s, \zeta)=(s, \zeta_1, \cdots, \zeta_n)\in \mathbb{R}^{1+n}.
$$

Quasi-linear parabolic equations
may  describe many   diffusion  phenomena in nature.
Suppose
  that a
  body   $\Omega$  is uniform and anisotropic.
 Denote by $y$, $c$ and $\rho$
 the temperature distribution, the specific heat and the
  density
   of the
  body, respectively. Then
 the conservation law  gives
$
c\rho y_t+\text{div}\, {\bf J}+f=u,
$
where  ${\bf J}$ is the heat flux vector, $\text{div}\, {\bf J}$ denotes its divergence, and
$f$  and $u$ are the heat sources (of some different kinds).
By the Fourier experiment law,  the heat  flux  vector takes the form of   ${\bf J}=
-\mathcal{A}\nabla y$
with a thermal conductivity coefficient $\mathcal{A}=(a^{i j})_{1\leq i, j\leq n}$.
When the thermal conductivity of material of the body $\Omega$
depends on temperature,  $\mathcal{A}$ has the form of $a^{i j}=a^{i j}(y)$.
For example,
the thermal conductivity of copper and aluminum
 generally decreases as temperature increases.
Further, when the heat source $f$
 is generated by certain chemical reactions, it usually  also depends  on
 the temperature distribution $y$.  We assume that  the  heat source
 $u$  can be artificially regulated  on a local domain $\omega$. Then the thermal  diffusion
 process
  may be described as
\begin{equation}\label{DD1}
c \rho y_t-\sum\limits_{i,j=1}^{n}(a^{i j}(y)y_{x_i})_{x_j}+f(y)=\chi_{\omega} u.
\end{equation}
 When  $y$ is sufficiently small, $f(y)$
 takes an approximate value of zero.
 Further, if we set $\widehat{a}^{i j}=a^{i j}(0)$ and
 $\widetilde{a}^{i j}(\cdot)=a^{i j}(\cdot)-\widehat{a}^{i j}$,
 then $\widetilde{a}^{i j}(0)=0$.
 Hence, formally,
the quasi-linear parabolic equation (\ref{DD1})
can be approximately  replaced by the following linear parabolic equation:
\begin{equation}\label{DD2}
c \rho y_t-\sum\limits_{i,j=1}^{n}(\widehat{a}^{i j} y_{x_i})_{x_j}=\chi_{\omega} u.
\end{equation}
For simplicity of presentation,   we assume in the sequel  that
   $c\rho=1$; and $\widehat{a}^{i j}=1$ for $i=j$ and $\widehat{a}^{i j}=0$ for $i\neq j$. In such a case, (\ref{DD1}) is
    indeed the first equation in (\ref{aeq11}), and
   (\ref{DD2}) becomes the classical controlled heat  conduction equation:
   \begin{eqnarray}\label{aeq1*}
	\left\{
	\begin{array}{ll}
		w_t-\Delta w=\chi_{\omega} v &\mbox{ in }Q, \\ \ns
		w=0 &\mbox{ on }\Sigma, \\ \ns
		w(x, 0)=y_0(x) &\mbox{ in }\Omega. \\ \ns
	\end{array}
	\right.
\end{eqnarray}
Here, we use $w$ and $v$ in (\ref{aeq1*}) to replace, respectively,  $y$ and $u$ in (\ref{aeq11}),
in order to avoid confusion of notations.

\smallskip

First, we recall  the following local well-posedness results for
strong solutions to (\ref{aeq11}).  By the  Schauder fixed point theorem and
the $L^p$-estimates for linear parabolic equations,
for any $p>n+2$,  there exists a positive constant $\rho_1$,
such that for any $u\in L^p(\omega\times(0, T))$ and $y_0\in W^{2, p}(\Omega)\cap H^1_0(\Omega)$
satisfying that
$$
|u|_{L^p(\omega\times(0, T))}\leq \rho_1\quad\mbox{and}\quad
|y_0|_{W^{2, p}(\Omega)}\leq \rho_1,
$$
the quasi-linear parabolic equation (\ref{aeq11})
admits a unique strong solution $y\in W^{2,1}_p(Q)\cap C^{1, 0}(\overline{Q})$. Moreover,
\begin{equation}\label{App1}
|y|_{C^{1, 0}(\overline{Q})}+|y|_{W^{2,1}_p(Q)}\leq \mathcal{C}(|u|_{L^p(\omega\times(0, T))}+|y_0|_{W^{2, p}(\Omega)}),
\end{equation}
where $\mathcal{C}$ denotes a positive constant, depending only on
$n$, $T$, $\Omega$, $\omega$, $p$, $a^{i j}$ and $f$, which may be different from one
place to another.
In the sequel, in order to ensure the local well-posedness of (\ref{aeq11}),
we  always assume  $\rho_1<1$ and  the initial value
$y_0\in W^{2, p}(\Omega)\cap H^1_0(\Omega)$
with $
|y_0|_{W^{2, p}(\Omega)}\leq \rho_1$,  and set $\mathcal{U}_{ad}=\Big\{\ u\in L^p(\omega\times(0, T))\ \big|\ |u|_{L^p(\omega\times(0, T))}\leq\rho_1\ \Big\}$.

\smallskip

Further,  define the following quadratic  cost functionals  on $\mathcal{U}_{ad}$:
$$
J_1(u(\cdot))=\frac{\alpha}{2}\displaystyle\int_Q |y(x, t)-y_d(x, t)|^2 dxdt+
\frac{\beta}{2}\displaystyle\int^T_0\int_{\omega} |u(x, t)|^2 dxdt,
$$
and
$$
J_2(v(\cdot))=\frac{\alpha}{2}\displaystyle\int_Q |w(x, t)-y_d(x, t)|^2 dxdt+
\frac{\beta}{2}\displaystyle\int^T_0\int_{\omega} |v(x, t)|^2 dxdt,
$$
where $\alpha$ and $\beta$ are two positive constants, $y_d\in L^{p}(Q)$ is a given
target,
and $y$ and $w$ are the  solutions to (\ref{aeq11}) and (\ref{aeq1*}) associated to $u$ and $v$,
respectively.

\smallskip

We consider the following two optimal control problems:
$$
{\bf (P_1)} \quad \min\limits_{u(\cdot)\in \mathcal{U}_{ad}} J_1(u(\cdot))
\quad\quad\mbox{ and }\quad\quad
{\bf (P_2)} \quad \min\limits_{v(\cdot)\in \mathcal{U}_{ad}} J_2(v(\cdot)).
$$
The functions $\overline{u}$ and $\overline{v}$ in $\mathcal{U}_{ad}$
are called optimal controls, respectively,  for the optimal control problems ${\bf (P_1)}$ and  ${\bf (P_2)}$, if
$$
J_1(\overline{u})=\min\limits_{u(\cdot)\in \mathcal{U}_{ad}} J_1(u(\cdot))
\quad\quad\mbox{ and }\quad\quad
J_2(\overline{v})=\min\limits_{v(\cdot)\in \mathcal{U}_{ad}} J_2(v(\cdot)).
$$
Indeed,  ${\bf (P_2)}$ is an LQ problem and  has a unique optimal control $\overline{v}$. Meanwhile,   there  exists at
least an  optimal control $\overline{u}$ for  ${\bf (P_1)}$, whose proof will be given in the Appendix.
Denote by $\overline{y}$ and $\overline{w}$ the optimal trajectories to
 (\ref{aeq11}) and (\ref{aeq1*}), respectively,  associated to the optimal controls $\overline{u}$ and $\overline{v}$.

\smallskip

First,  we give a verification condition
 determining whether the optimal control is an interior point or a boundary point of
 $\mathcal{U}_{ad}$.

\begin{proposition}\label{APR}
The following results hold for the optimal control problems ${\bf (P_1)}$ and  ${\bf (P_2)}$.

\noindent $(1)$  There exists a  positive constant $\rho^*_1$,
such that when $\beta/\alpha>\rho^*_1$,
the optimal controls $\overline{u}$ and  $\overline{v}$
are the interior points of  $\mathcal{U}_{ad}$;

\noindent $(2)$ Assume that either $y_d(x, t)\geq d_0$ or $y_d(x, t)\leq -d_0$ in $Q$ for
some positive constant $d_0$. Then, for each sufficiently small $\rho_1$,
 there exists a  positive constant $\rho^*_2$,
such that when $\beta/\alpha\leq\rho^*_2$,
the optimal controls $\overline{u}$ and  $\overline{v}$
are the boundary points of  $\mathcal{U}_{ad}$.
\end{proposition}

The  purpose of this paper is  to  establish a connection between   the optimal control problems ${\bf (P_1)}$ and  ${\bf (P_2)}$.  More precisely,
  we shall show that,    when the initial values $y_0$ and
  targets $y_d$ in the cost functionals are  sufficiently small in  the suitable Sobolev spaces,  the gap between the optimal trajectories
for the quasi-linear and linear parabolic equations
is a high-order  infinitesimal with respect to these given data.
  This fact  indicates that
the study of the optimal control problem of (\ref{aeq11})
  can be replaced by the LQ problem of (\ref{aeq1*}), and the errors  between them may be ignored.
In the literature, there exist extensive works addressing to optimal control problems of
 parabolic equations (e.g., \cite{Bar}-\cite{H}, \cite{Li}-\cite{Wang},   and the rich references therein).
  However, to the best of our knowledge,
the approximation relationship on optimal control problems between nonlinear
and linear partial differential equations  has never been considered
before.

\medskip

 First,   we introduce the following assumption on  $f$:
\begin{eqnarray*}
\begin{array}{ll}
&\displaystyle
{\bf (H) }\quad  f(0)=f'(0)=0.
 \end{array}
\end{eqnarray*}
Then we have the following approximation result between the quasi-linear parabolic equation (\ref{aeq11}) and  the linear heat equation (\ref{aeq1*})
(with the same controls and initial values).

\begin{proposition}\label{prop1}
Assume that ${\bf (H)}$ holds. Then for any $y_0\in W^{2, p}(\Omega)\cap H^1_0(\Omega)$
and $u\in L^p(\omega\times(0, T))$ satisfying that
 $$
 |y_0|_{W^{2, p}(\Omega)}+|u|_{L^p(\omega\times(0, T))}\leq \rho_1,
 $$
 the following estimate holds for the corresponding solutions $y$ and $w$
 to $(\ref{aeq11})$ and $(\ref{aeq1*})$ with $(v=u)$:
$$
|y-w|_{W^{2,1}_p(Q)}\leq
C\big(|y_0|^2_{W^{2, p}(\Omega)}+
|u|^{2}_{L^p(\omega\times(0, T))}
\big),
$$
here and hereafter,  $C$ denotes a positive constant, depending only on
$n$, $T$, $\Omega$,  $\omega$, $p$, $a^{i j}$, $f$, $\alpha$ and $\beta$, which may be different from one
place to another.

\end{proposition}

Proposition \ref{prop1}
indicates that
when the initial values and controls are sufficiently small,
the gap between the  solutions to the quasi-linear
parabolic equation (\ref{aeq11}) and the linear parabolic equation (\ref{aeq1*})
is a high-order infinitesimal with respect to  given small data.
The proof of Proposition \ref{prop1} is based on the
 $L^p$-estimates for linear parabolic equations, and it will be given in the Appendix.

\medskip

 The first main result of this paper on  optimal control problems  is as follows.
\begin{theorem}\label{th2}
Assume that ${\bf (H)}$  holds.
Then there exist two constants $\rho_2, \theta\in (0, 1)$, such that
 for any   $y_0\in W^{2, p}(\Omega)\cap H^1_0(\Omega)$
 and $y_d\in L^{p}(Q)$ satisfying that
 $$
 |y_0|_{W^{2, p}(\Omega)}+|y_d|_{L^p(Q)}\leq \rho_2,
 $$
 the following estimate holds for the optimal pairs $(\overline{u}, \overline{y})$ and $(\overline{v}, \overline{w})$
 of the optimal control problems ${\bf (P_1)}$ and ${\bf (P_2)}$:
$$
|\overline{u}-\overline{v}|_{L^2(\omega\times(0, T))}+
|\overline{y}-\overline{w}|_{W^{2,1}_2(Q)}\leq
C\big(|y_d|_{L^p(Q)}+|y_0|_{W^{2, p}(\Omega)}\big)^{1+\theta}.
$$
\end{theorem}

Theorem \ref{th2}  gives  the quantitative comparison between the
 optimal control problems for the quasi-linear parabolic equation
  (\ref{aeq11}) and the linear heat equation (\ref{aeq1*}).
 The optimal controls
$\overline{u}$ and $\overline{v}$  not only respectively optimize the cost functionals
$J_1$ and $J_2$,
 but also  the errors between  their associated  solutions  are  a high-order infinitesimal
  with respect to given small  data.  This indicates that, in some sense one may replace the optimal control  problem ${\bf(P_1)}$ for the quasi-linear parabolic equation (\ref{aeq11}) by an
  LQ problem.

\smallskip

On the other hand,  in the case where the optimal controls are  interior points of $\mathcal{U}_{ad}$, we have the following
sharp approximation result.
\begin{theorem}\label{th3}
Assume that ${\bf (H)}$  holds, and $\overline{u}$ and  $\overline{v}$
are interior points of $\mathcal{U}_{ad}$.
Then there exists a  constant $\rho_3\in (0, 1)$, such that
 for any   $y_0\in W^{2, p}(\Omega)\cap H^1_0(\Omega)$
 and $y_d\in L^{p}(Q)$ satisfying that
 $$
 |y_0|_{W^{2, p}(\Omega)}+|y_d|_{L^p(Q)}\leq \rho_3,
 $$
 the following estimate holds for the optimal pairs $(\overline{u}, \overline{y})$ and $(\overline{v}, \overline{w})$
 of the optimal control problems ${\bf (P_1)}$ and ${\bf (P_2)}$:
$$
|\overline{u}-\overline{v}|_{L^p(\omega\times(0, T))}+
|\overline{y}-\overline{w}|_{W^{2,1}_p(Q)}\leq
C\big(|y_d|^{2}_{L^p(Q)}+|y_0|^2_{W^{2, p}(\Omega)}\big).
$$
\end{theorem}

\begin{remark}
The second-order approximation obtained in Theorem $\ref{th3}$  is optimal in
general. The key reason is that the linear parabolic equation is not an
arbitrarily chosen approximate model, but rather the first-order linearization
of the quasi-linear equation around the zero equilibrium. In other words, when
the initial datum, the target, and the corresponding optimal control are all
small, the quasi-linear system and the linear system have exactly the same
first-order response. Therefore, no first-order error appears between the two
systems.

Under the assumption ${\bf(H)}$, the nonlinear source term $f$ has neither a
constant term nor a linear term at the origin, while the difference between the
diffusion coefficient $a^{ij}(y)$ and its value $a^{ij}(0)$ enters the equation
only after being multiplied by the state gradient. Hence, the first genuine
difference between the quasi-linear equation and its linearized equation is of
quadratic order. Intuitively, if the size of the initial datum and the target
is denoted by $\delta$, then the linear system captures all leading variations
of order $\delta$, whereas the quasi-linear effects first appear at order
$\delta^2$.

This second-order error cannot generally be removed. Unless additional
degeneracy or special cancellation conditions are imposed on the nonlinear
terms, for instance requiring the relevant quadratic nonlinear contributions to
vanish identically, these quadratic terms will genuinely affect the state
equation, the adjoint equation, and hence the optimal control determined by
the optimality condition. Thus, in the general quasi-linear setting, one cannot
expect the difference between the quasi-linear optimal pair and the linear
optimal pair to be of order higher than two.

In this sense, the estimate in Theorem $\ref{th3}$ shows not only that the
linear optimal control problem provides an effective approximation of the
quasi-linear optimal control problem, but also that this approximation reaches
the best order allowed by the nonlinear structure of the equation. That is,
order two is the optimal approximation order one can expect in general.

In the following, we give an example to illustrate this point.  Assume that $\Omega=(0, \pi)$
and $\omega=\Omega$.  Let $\delta > 0$ be a small parameter representing
the magnitude of the data, and define $y_0(x) = \delta \sin x$ and $y_d(x, t) = \delta \sin x.$
Further,  we assume that $\alpha=1$,
   $a^{i j}=1$ for $i=j$ and $a^{i j}=0$ for $i\neq j$, and $f(s)=s^2$. Then, $(\ref{aeq11})$ and
   $(\ref{aeq1*})$ become, respectively, the following semi-linear and linear parabolic equations:
\begin{eqnarray}\label{!aeq11}
	\left\{
	\begin{array}{ll}
		y_t-y_{xx}+y^2= u &\mbox{ in }Q, \\ \ns
		y=0 &\mbox{ on }\Sigma, \\ \ns
		y(x, 0)=\delta \sin x &\mbox{ in }\Omega,
	\end{array}
	\right.
\end{eqnarray}
and
 \begin{eqnarray}\label{!aeq1*}
	\left\{
	\begin{array}{ll}
		w_t-w_{xx}=v &\mbox{ in }Q, \\ \ns
		w=0 &\mbox{ on }\Sigma, \\ \ns
		w(x, 0)=\delta \sin x  &\mbox{ in }\Omega.
	\end{array}
	\right.
\end{eqnarray}
By Proposition $\ref{APR}$,
when  $\beta$ is sufficiently large,  the optimal controls $\overline{u}$ and  $\overline{v}$
 are interior points of the admissible control set.  By Theorem $\ref{th3}$,
 $$
 |\overline{u}-\overline{v}|_{L^2(\Omega\times(0, T))}\leq C\delta^2.
 $$

Now, we will prove that $|\overline{u}-\overline{v}|_{L^2(\Omega\times(0, T))}\geq C_0\delta^2$ for a
positive constant $C_0$, which implies the optimality of  order $2$. To this aim,   first, for the linear system
$(\ref{!aeq1*})$,  notice that
$$
\overline{v}=\delta v_1, \quad \overline{w}=\delta w_1,\quad  z=\delta z_1,\quad \mbox{ and }\quad  v_1=\beta^{-1} z_1,
$$
where $v_1$, $w_1$ and $z_1$ are the optimal control, optimal trajectory and solution to
 the adjoint equation $z_{1, t}+z_{1, xx}=w_1-\sin x$, respectively, associated to $\delta=1$.

For the semi-linear system $(\ref{!aeq11})$, set
$$
\overline{u}=\delta u_1+\delta^2 u_2+O(\delta^3),\
\overline{y}=\delta y_1+\delta^2 y_2+O(\delta^3)
\mbox{ and }
\eta=\delta \eta_1+\delta^2 \eta_2+O(\delta^3),\
$$
where $O(\delta^3)$ denotes the infinitesimals of the same order with $\delta^3$, and
$\eta$ is the solution to the adjoint equation $\eta_t+\eta_{xx}-2\overline{y}\eta=\overline{y}-\delta \sin x$.
Substituting the above expansions into the optimality system of the optimal control
problem ${\bf(P_1)}$ and  matching the terms of order $\delta$,
we have that
$y_{1, t}-y_{1, xx}= u_1$ with $y_1(x, 0)=\sin x$,
$\eta_{1, t}+\eta_{1, xx}=y_1-\sin x$ with $\eta_1(x, T)=0$,
and $u_1=\beta^{-1}\eta_1$.
Hence, it follows that
$$
u_1=v_1,\quad y_1=w_1\quad \mbox{ and } \quad \eta_1=z_1,
$$
and the first-order error between $\overline{u}$ and $\overline{v}$ vanishes.
This implies that $\overline{u}-\overline{v}=\delta^2 u_2+O(\delta^3)=\delta^2\beta^{-1}\eta_2+O(\delta^3)$.

It suffices to prove that $|\eta_2|_{L^2(\Omega\times(0, T))}> 0$.
 By matching the terms of order $\delta^2$ in the optimality system of ${\bf(P_1)}$,
we have that
$y_{2, t}-y_{2, xx}=\beta^{-1}\eta_2-y_1^2$ with $y_2(x, 0)=0$,  and
$\eta_{2, t}+\eta_{2, xx}=y_2+2y_1\eta_1$ with $\eta_2(x, T)=0$.
By contradiction, we assume that $\eta_2=0$ in $Q$.
Then, $y_{2, t}-y_{2, xx}=-y^2_1$ in $Q$ and $y_2(x, 0)=0$.  Notice that when
 $\rho$ is sufficiently small,
$y_1>0$ in $(\pi/2, 3\pi/4)\times (0, T]$ and therefore,
$|y_2(\cdot, T)|_{L^2(\Omega)}>0$. This implies that
\begin{equation}\label{AA}
\lim\limits_{t \rightarrow T^-}
\big[ y_2(x,t) + 2 y_1(x,t) \eta_1(x,t) \big] =  y_2(x,T)\not\equiv 0.
\end{equation}

On the other hand, since $\eta_2=0$ and $\eta_{2, t}+\eta_{2, xx}=y_2+2y_1\eta_1$ in $Q$,  we obtain that
$$
y_2(x, t)+2y_1(x, t)\eta_1(x, t)=0 \quad \mbox{ in }\ Q,
$$
which contradicts $(\ref{AA})$.
\end{remark}

\smallskip

Furthermore,    if the order in the cost functional $J_1$ is greater than quadratic,  similar results still hold.
To this aim, we consider the following cost functionals  on $\mathcal{U}_{ad}$ for $\hat{p}>2$:
$$
\widehat{J}_1(u(\cdot))=\frac{\alpha}{\hat{p}}\displaystyle\int_Q |y(x, t)-y_d(x, t)|^{\hat{p}} dxdt+
\frac{\beta}{2}\displaystyle\int^T_0\int_{\omega} |u(x, t)|^{2} dxdt,
$$
and
$$
\widehat{J}_2(v(\cdot))=\frac{\alpha}{\hat{p}}\displaystyle\int_Q |w(x, t)-y_d(x, t)|^{\hat{p}}  dxdt+
\frac{\beta}{2}\displaystyle\int^T_0\int_{\omega} |v(x, t)|^2 dxdt,
$$
where  $y$ and $w$ are the  solutions to (\ref{aeq11}) and (\ref{aeq1*}) associated to $u$ and $v$,
respectively.
At this moment, we consider the following optimal control problems:
$$
{\bf (\widehat{P}_1)} \quad \min\limits_{u(\cdot)\in \mathcal{U}_{ad}} \widehat{J}_1(u(\cdot))
\quad\quad\mbox{ and }\quad\quad
{\bf (\widehat{P}_2)} \quad \min\limits_{v(\cdot)\in \mathcal{U}_{ad}} \widehat{J}_2(v(\cdot)).
$$
Then  the optimal control problem ${\bf (\widehat{P}_1)}$ has at least an optimal control $u^*$ in $\mathcal{U}_{ad}$,
 and  ${\bf (\widehat{P}_2)}$ has
a unique optimal control
 $v^*$ in $\mathcal{U}_{ad}$, that is,
$$
\widehat{J}_1(u^*)=\min\limits_{u(\cdot)\in \mathcal{U}_{ad}} \widehat{J}_1(u(\cdot))
\quad\quad\mbox{ and }\quad\quad
\widehat{J}_2(v^*)=\min\limits_{v(\cdot)\in \mathcal{U}_{ad}} \widehat{J}_2(v(\cdot)).
$$

The connection of the optimal control problem ${\bf(\widehat{P}_1)}$ and
its linear counterpart  ${\bf(\widehat{P}_2)}$ is stated as follows.

\begin{theorem}\label{th4}
Under the
assumption ${\bf(H)}$,
there exists a  constant $\rho_4\in (0, 1)$,  such that
 for any   $y_0\in W^{2, p}(\Omega)\cap H^1_0(\Omega)$
 and $y_d\in L^{\infty}(Q)$ satisfying that
 $$
 |y_0|_{W^{2, p}(\Omega)}+|y_d|_{L^\infty(Q)}\leq \rho_4,
 $$
 the following estimate holds for the optimal pairs $(u^*, y^*)$ and $(v^*, w^*)$
 of the optimal control problems ${\bf (\widehat{P}_1)}$ and ${\bf (\widehat{P}_2)}$:
$$
|u^*-v^*|_{L^p(\omega\times(0, T))}+
|y^*-w^*|_{W^{2,1}_p(Q)}\leq
C\big(|y_d|^{2}_{L^\infty(Q)}+|y_0|^2_{W^{2, p}(\Omega)}\big),
$$
with $\rho_1$ in $\mathcal{U}_{ad}$ being replaced by $\rho_4$.
\end{theorem}

\begin{remark}
By virtue of similar arguments in this paper, we can study the approximation of optimal control problems
between the following general quasi-linear parabolic equation and the linear counterpart:
$$
y_t-\sum\limits_{i,j=1}^{n}(a^{i j}(y, \nabla y)y_{x_i})_{x_j}+f(y, \nabla y)=\chi_{\omega} u.
$$
Further, the  approximation relationship on optimal control problems between quasi-linear hyperbolic and
linear wave equations can also be proved.
These results will be given in a forthcoming paper.
\end{remark}

The rest of this paper is organized as follows.
Section \ref{chapter2} presents verification
 criteria for determining whether an optimal control lies in the interior or on the boundary of the admissible control set. Section \ref{chapter3}  analyzes the relationship between the optimal
 control problems governed by the quasi-linear and linear parabolic equations.
 Section  \ref{chapter4} discusses two cases in which the approximation error
  is of order two. Finally, the Appendix proves the existence of optimal controls,
   and establishes the approximation result between the quasi-linear and linear
   parabolic equations with the same control and initial datum.

\section{Proof of Proposition \ref{APR}}\label{chapter2}

In this section, we prove  a verification condition   ensuring
  the optimal control to be an interior point or a boundary point of
 $\mathcal{U}_{ad}$.

 \smallskip

\noindent {\bf Proof of Proposition \ref{APR}. }  The whole proof is divided into three parts.

\smallskip

{\bf Step 1. } First,  we give the necessary conditions satisfied by the optimal controls
$\overline{u}$ and $\overline{v}$ for the problems ${\bf (P_1)}$ and ${\bf (P_2)}$.

\smallskip

For
the optimal control $\overline{u}\in \mathcal{U}_{ad}$ of the problem ${\bf (P_1)}$,
 any $u\in \mathcal{U}_{ad}$ and $\varepsilon\in(0, 1)$, set
$u_\varepsilon=\overline{u}+\varepsilon(u-\overline{u})$. Then
$u_\varepsilon\in \mathcal{U}_{ad}$ and
\begin{equation}\label{c1}
\lim\limits_{\varepsilon\rightarrow 0}
\frac{J_1(u_\varepsilon)-J_1(\overline{u})}{\varepsilon}=
\alpha\int_Q (\overline{y}-y_d)\xi dxdt+
\beta\int_0^T\int_{\omega} \overline{u}(u-\overline{u}) dxdt\geq 0,
\end{equation}
where $\xi\in L^2(0, T; H^1_0(\Omega))\cap
C([0, T]; L^2(\Omega))$ satisfies the following linear parabolic equation:
\begin{eqnarray}\label{c2}
	\left\{
	\begin{array}{ll}
		\xi_t-\sum\limits_{i, j=1}^{n}\big(a^{ij}(\overline{y})\xi_{x_i}\big)_{x_j}-
		\sum\limits_{i, j=1}^{n}\big[(a^{i j})'(\overline{y})\overline{y}_{x_i}\xi\big]_{x_j}+f'(\overline{y})\xi=\chi_{\omega}(u-\overline{u})	
		&\mbox{ in }Q, \\ \ns
		\xi=0 &\mbox{ on }\Sigma, \\ \ns
		\xi(x, 0)=0 &\mbox{ in }\Omega. \\ \ns
	\end{array}
	\right.
\end{eqnarray}

Next, introduce the following linear parabolic equation:
\begin{eqnarray}\label{c3}
	\left\{
	\begin{array}{ll}
		\eta_t+\sum\limits_{i, j=1}^{n}\big(a^{ij}(\overline{y})\eta_{x_i}\big)_{x_j}-
		\sum\limits_{i, j=1}^{n} (a^{i j})'(\overline{y})\overline{y}_{x_i}\eta_{x_j}-f'(\overline{y})\eta=\alpha(\overline{y}-y_d)
		&\mbox{ in }Q, \\ \ns
		\eta=0 &\mbox{ on }\Sigma, \\ \ns
		\eta(x, T)=0 &\mbox{ in }\Omega. \\ \ns
	\end{array}
	\right.
\end{eqnarray}
In (\ref{c3}),  since $\overline{y}\in W^{2, 1}_p(Q)$ and $p>n+2$,
we have that for $i, j=1, \cdots, n$,
$$
\overline{y},  a^{i j}(\overline{y})\in C^{1, 0}(\overline{Q}),  \quad
 (a^{i j})'(\overline{y})\overline{y}_{x_i}, f'(\overline{y})\in C(\overline{Q}),\quad
 \mbox{and} \quad  \overline{y}-y_d
	 \in L^p(Q).
$$
Hence, (\ref{c3}) admits a unique  strong solution
$\eta\in W^{2,1}_p(Q)$ (see \cite{Lady}).
Moreover,
\begin{equation}\label{new1}
|\eta|_{W^{2, 1}_p(Q)}\leq C_1\alpha |\overline{y}-y_d|_{L^p(Q)},
\end{equation}
where $C_1$ denotes a positive constant, depending only on
$n$, $T$, $\Omega$,  $\omega$, $p$, $a^{i j}$ and $f$.
 Multiplying both sides of  the first equation in (\ref{c2}) by $\eta$ and integrating it
on $Q$, we obtain that
$$
\alpha\int_Q (\overline{y}-y_d) \xi dxdt+
\int_0^T\int_{\omega} (u-\overline{u})\eta dxdt=0.
$$
This, together with (\ref{c1}), indicates that
\begin{equation}\label{c5}
\int^T_0\int_\omega (\beta\overline{u}-\eta)(u-\overline{u})
dxdt\geq 0,
\quad \forall \ u\in \mathcal{U}_{ad}.
\end{equation}

On the other hand, using the similar arguments above  for the optimal control problem ${\bf (P_2)}$ and recalling that $(\overline{v}, \overline{w})$ is the unique optimal pair,  we get that
\begin{equation}\label{PP1}
\displaystyle\int^T_0\int_\omega
(\beta\overline{v}-z)(u-\overline{v})dxdt\geq 0,
\quad \forall \ u\in \mathcal{U}_{ad},
\end{equation}
where $z$ satisfies the linear parabolic equation:
\begin{eqnarray}\label{c7}
	\left\{
	\begin{array}{ll}
		z_t+\Delta z=\alpha (\overline{w}-y_d)	
		&\mbox{ in }Q, \\ \ns
		z=0 &\mbox{ on }\Sigma, \\ \ns
		z(x, T)=0 &\mbox{ in }\Omega. \\ \ns
	\end{array}
	\right.
\end{eqnarray}
Similar to  (\ref{new1}),  we also have that
\begin{equation}\label{new1*}
|z|_{W^{2, 1}_p(Q)}\leq C_1\alpha |\overline{w}-y_d|_{L^p(Q)}.
\end{equation}

\smallskip

{\bf Step 2. } In this step,  we give the expressions of the optimal controls
$\overline{u}$ and $\overline{v}$.

\smallskip

{\bf Case 1. }  Assume that $|\overline{u}|_{L^p(\omega\times(0, T))}<\rho_1$
and $|\overline{v}|_{L^p(\omega\times(0, T))}<\rho_1$, i.e.,
$\overline{u}$ and $\overline{v}$ are two  interior points of $\mathcal{U}_{ad}$.

\medskip

In this case,  by (\ref{c5}) and (\ref{PP1}), it follows that
\begin{equation}\label{Add}
\displaystyle \overline{u}=\chi_\omega \beta^{-1} \eta\ \ \quad
\mbox{ and } \ \  \quad \overline{v}=\chi_\omega \beta^{-1} z
\quad\mbox{ in }\ Q.
\end{equation}

\medskip

{\bf Case 2. }  Assume that $|\overline{u}|_{L^p(\omega\times(0, T))}=\rho_1$
and $|\overline{v}|_{L^p(\omega\times(0, T))}=\rho_1$, i.e.,
$\overline{u}$ and $\overline{v}$ are   boundary points of $\mathcal{U}_{ad}$.
By (\ref{c5}),
\begin{equation}\label{lwz1}
\displaystyle\int^T_0\int_\omega  (\beta\overline{u}-\eta)u
dxdt\geq
\int^T_0\int_\omega  (\beta\overline{u}-\eta)\overline{u}dxdt,
\ \ \forall \ u\in L^p(\omega\times(0, T)) \mbox{ with } |u|_{L^p(\omega\times(0, T))}
\leq\rho_1.
\end{equation}
Notice that for $q=\displaystyle\frac{p}{p-1}$ $(\mbox{i.e.},\ \displaystyle\frac{1}{p}+\frac{1}{q}=1)$,
$$
|\beta\overline{u}-\eta|_{L^q(\omega\times(0, T))}
=\sup\limits_{|v|_{L^p(\omega\times(0, T))}=1}
\Big|\displaystyle\int^T_0\int_\omega  (\beta\overline{u}-\eta)v
dxdt\Big|.
$$
Hence, for any $v\in L^p(\omega\times(0, T))$ with
$|v|_{L^p(\omega\times(0, T))}=1$,
$$
\displaystyle\int^T_0\int_\omega  (\beta\overline{u}-\eta)v
dxdt\geq -|\beta\overline{u}-\eta|_{L^q(\omega\times(0, T))}.
$$
Moreover, there exists a $v^*_1\in L^p(\omega\times(0, T))$ with
$|v^*_1|_{L^p(\omega\times(0, T))}=1$, such that
$$
\displaystyle\int^T_0\int_\omega  (\beta\overline{u}-\eta)v_1^*
dxdt= -|\beta\overline{u}-\eta|_{L^q(\omega\times(0, T))}.
$$
Indeed,  $v^*_1$ can be chosen as
$$
v^*_1(x, t)=
\left\{
\begin{array}{lll}
&\displaystyle\frac{-1}{|\beta\overline{u}-\eta|^{q-1}_{L^q(\omega\times(0, T))}}\cdot
\frac{\beta\overline{u}(x, t)-\eta(x, t)}{|\beta\overline{u}(x, t)-\eta(x, t)|^{2-q}}, &\mbox{ if }\  \beta\overline{u}(x, t)\neq\eta(x, t);\\[6mm]
&\displaystyle 0, &\mbox{ if }\ \beta\overline{u}(x, t)=\eta(x, t).
\end{array}
\right.
$$

This, together with (\ref{lwz1}), indicates that the infimum of the functional  $\tilde{J}(v(\cdot))=\displaystyle\int^T_0\int_\omega (\beta\overline{u}-\eta)v
dxdt$ can be attained on the unit sphere of $L^p(\omega\times(0, T))$, and therefore,
\begin{equation}\label{lwz2}
-|\beta\overline{u}-\eta|_{L^q(\omega\times(0, T))}
=\int^T_0\int_\omega (\beta\overline{u}-\eta)\frac{\overline{u}}{\rho_1}dxdt.
\end{equation}

It follows that
$$\displaystyle |\beta\overline{u}-\eta|_{L^q(\omega\times(0, T))}
=\Big|\int^T_0\int_\omega(\beta\overline{u}-\eta)\frac{\overline{u}}{\rho_1}dxdt\Big|
\leq |\beta\overline{u}-\eta|_{L^q(\omega\times(0, T))}
\Big|\frac{\overline{u}}{\rho_1}\Big|_{L^p(\omega\times(0, T))}
\!=|\beta\overline{u}-\eta|_{L^q(\omega\times(0, T))}.
$$
Hence,
$$\Big|\int^T_0\int_\omega(\beta\overline{u}-\eta)\frac{\overline{u}}{\rho_1}dxdt\Big|
= |\beta\overline{u}-\eta|_{L^q(\omega\times(0, T))}
\Big|\frac{\overline{u}}{\rho_1}\Big|_{L^p(\omega\times(0, T))}.
$$
This means that the equal sign in the H\"older inequality holds, which is equivalent to the following fact:
$$
\beta\overline{u}-\eta=\lambda \Big|\frac{\overline{u}}{\rho_1}\Big|^{p-2}
\frac{\overline{u}}{\rho_1}\quad\quad\mbox{a.e. \ \ in } \ \omega\times(0, T),
$$
where $\lambda$ is a constant.
By (\ref{lwz2}), we obtain that
\begin{eqnarray*}
&&-|\beta\overline{u}-\eta|_{L^q(\omega\times(0, T))}=
\displaystyle\int^T_0\int_{\omega}
\lambda \Big|\frac{\overline{u}}{\rho_1}\Big|^{p}dxdt=\lambda,
\end{eqnarray*}
and therefore,
$$
\beta\overline{u}-\eta=-|\beta\overline{u}-\eta|_{L^q(\omega\times(0, T))} \Big|\frac{\overline{u}}{\rho_1}\Big|^{p-2}
\frac{\overline{u}}{\rho_1}\quad\quad\mbox{a.e. \ \ in } \ \omega\times(0, T).
$$
This indicates that
\begin{equation}\label{lwz4}
\eta=\beta\overline{u}+|\beta\overline{u}-\eta|_{L^q(\omega\times(0, T))} \Big|\frac{\overline{u}}{\rho_1}\Big|^{p-2}
\frac{\overline{u}}{\rho_1}\quad\quad\mbox{a.e. \ \ in } \ \omega\times(0, T).
\end{equation}

Similarly, we also have that
\begin{equation}\label{lwz4**}
z=\beta\overline{v}+|\beta\overline{v}-z|_{L^q(\omega\times(0, T))} \Big|\frac{\overline{v}}{\rho_1}\Big|^{p-2}
\frac{\overline{v}}{\rho_1}\quad\quad\mbox{a.e. \ \ in } \ \omega\times(0, T).
\end{equation}

{\bf Step 3. }  Suppose that $\overline{u}$ is   the boundary point of  $\mathcal{U}_{ad}$,   by (\ref{lwz4}),
\begin{equation}\label{lwz8}
|\eta|_{L^p(\omega\times(0, T))}\geq\beta|\overline{u}|_{L^p(\omega\times(0, T))}
=\beta\rho_1.
\end{equation}
Further, by (\ref{new1}) and (\ref{App1}),  it is easy to show that
\begin{eqnarray*}
&&|\eta|_{W^{2, 1}_p(Q)}\leq C_1\alpha|\overline{y}-y_d|_{L^p(Q)}
\leq C_1\alpha \big(|\overline{u}|_{L^p(\omega\times(0, T))}+
|y_0|_{W^{2, p}(\Omega)}+
|y_d|_{L^p(Q)}\big)
\leq C_1\alpha \rho_1.
\end{eqnarray*}
Combining the above estimate with (\ref{lwz8}),  we obtain that there exists a positive constant
$\rho^*_1$,  such that
\begin{equation}\label{**1}
\beta\leq \rho_1^*\alpha.
\end{equation}
When $\displaystyle\frac{\beta}{\alpha}>\rho^*_1$,
there is a contradiction and therefore,
  $\overline{u}$ is   an interior point of  $\mathcal{U}_{ad}$.  Similar arguments still hold for
  the optimal control $\overline{v}$.

  \medskip

  In the following,  we assume that either  $y_d(x, t)\geq d_0$ or $y_d(x, t)\leq -d_0$  in $Q$ for some positive
   constant $d_0$ and
  $\overline{u}$ is   an interior point of  $\mathcal{U}_{ad}$.
   First,
by (\ref{App1}),
$$\displaystyle
|\overline{y}|_{C^{1, 0}(\overline{Q})}\leq
 C_1(|\overline{u}|_{L^p(\omega\times(0, T))}+|y_0|_{W^{2, p}(\Omega)})\leq C_1\rho_1.
$$
For a sufficiently small $\rho_1$, we have that
$\displaystyle
|\overline{y}|_{C^{1, 0}(\overline{Q})}\leq
 \frac{d_0}{2}.
$

\smallskip

{\bf Case 1) } Assume that   $y_d(x, t)\geq d_0$ in $Q$.
 It follows that $\displaystyle y_d(x, t)-\overline{y}(x, t)\geq  \frac{d_0}{2}
$ in $Q$.

 \smallskip

 In this case,  set $\tilde\eta(x, t)=\eta(x, T-t)$ and it satisfies the equation:
 \begin{eqnarray}\label{forwardc3}
	\left\{
	\begin{array}{ll}
		\tilde{\eta}_t-\sum\limits_{i, j=1}^{n}\big(a^{ij}(\overline{y}(x, T\!-\!t))\tilde{\eta}_{x_i}\big)_{x_j}+
		\sum\limits_{i, j=1}^{n} (a^{i j})'(\overline{y}(x, T\!-\!t))\overline{y}_{x_i}(x, T-t)\tilde{\eta}_{x_j}
		 &\\ \ns
		\quad\quad\quad\quad\quad\quad\quad\quad=-f'(\overline{y}(x, T\!-\!t))\tilde{\eta}+\alpha\big[y_d(x, T\!-\!t)-
		\overline{y}(x, T\!-\!t)\big]
		&\mbox{in }Q, \\ \ns
		\tilde{\eta}=0 &\mbox{on }\Sigma, \\ \ns
		\tilde{\eta}(x, 0)=0 &\mbox{in }\Omega. \\ \ns
	\end{array}
	\right.
\end{eqnarray}
 We construct a sub-solution $\phi(x, t)=\alpha M t\tilde{\psi}(x)$ to the equation (\ref{forwardc3}), where
 $\tilde{\psi}(\cdot)\in C^2(\Omega)$, $\tilde{\psi}=0$ on $\Gamma$,
 $\tilde{\psi}>0$ in $\Omega$, and $M$ is an undetermined positive constant.
 Indeed, it is easy to check that $\phi$ satisfies the equation:
 \begin{eqnarray*}
	\left\{
	\begin{array}{ll}
		\phi_t-\sum\limits_{i, j=1}^{n}\big(a^{ij}(\overline{y}(x, T\!-\!t))\phi_{x_i}\big)_{x_j}+
		\sum\limits_{i, j=1}^{n} (a^{i j})'(\overline{y}(x, T\!-\!t))\overline{y}_{x_i}(x, T-t)\phi_{x_j}
		+f'(\overline{y}(x, T\!-\!t))\phi &\\ \ns
		\quad=\alpha M\tilde{\psi}-\alpha Mt\Big[
		\sum\limits_{i, j=1}^{n}\big(a^{ij}(\overline{y}(x, T\!-\!t))\tilde{\psi}_{x_i}\big)_{x_j}&\\ \ns
		\quad\quad-
		\sum\limits_{i, j=1}^{n} (a^{i j})'(\overline{y}(x, T\!-\!t))\overline{y}_{x_i}(x, T-t)\tilde{\psi}_{x_j}
		-f'(\overline{y}(x, T\!-\!t))\tilde{\psi}
		\Big]
		&\mbox{in }Q, \\ \ns
		\phi=0 &\mbox{on }\Sigma, \\ \ns
		\phi(x, 0)=0 &\mbox{in }\Omega. \\ \ns
	\end{array}
	\right.
\end{eqnarray*}
Notice that
\begin{eqnarray*}
&&\alpha M\tilde{\psi}-\alpha Mt\Big[
		\sum\limits_{i, j=1}^{n}\big(a^{ij}(\overline{y}(x, T\!-\!t))\tilde{\psi}_{x_i}\big)_{x_j}\\[-2mm]
		&&-
		\sum\limits_{i, j=1}^{n} (a^{i j})'(\overline{y}(x, T\!-\!t))\overline{y}_{x_i}(x, T-t)\tilde{\psi}_{x_j}
		-f'(\overline{y}(x, T\!-\!t))\tilde{\psi}
		\Big]\\
		&&\leq C_1\alpha M+C_1\alpha MT,
\end{eqnarray*}
where  $C_1$  is a positive constant depending only on $a^{i j}$, $f$, $d_0$ and $\tilde{\psi}$, since
$\displaystyle
|\overline{y}|_{C^{1, 0}(\overline{Q})}\leq
 \frac{d_0}{2}.
$
Hence, there exists a sufficiently small positive constant $M$, depending only on $a^{i j}$, $f$, $d_0$, $\tilde{\psi}$ and $T$, such that
\begin{eqnarray*}
&&\displaystyle \alpha M\tilde{\psi}-\alpha Mt\Big[
		\sum\limits_{i, j=1}^{n}\big(a^{ij}(\overline{y}(x, T\!-\!t))\tilde{\psi}_{x_i}\big)_{x_j}\\
		&&-
		\sum\limits_{i, j=1}^{n} (a^{i j})'(\overline{y}(x, T\!-\!t))\overline{y}_{x_i}(x, T-t)\tilde{\psi}_{x_j}
		\!-\! f'(\overline{y}(x, T\!-\!t))\tilde{\psi}
		\Big]\leq \frac{\alpha d_0}{2}.
\end{eqnarray*}
By the comparison principle of linear parabolic equations,  $\phi\leq \tilde{\eta}$ in $Q$, i.e.,
$$\eta(x, t)\geq  \alpha M (T-t)\tilde{\psi}(x)\quad \mbox{ in }\ \ Q.
$$
This implies that
$$
|\eta|_{L^2(\omega\times(0, T))}\geq
|\alpha M (T-\cdot)\tilde{\psi}(\cdot)|_{L^2(\omega\times(0, T))}
= \displaystyle\frac{\alpha MT^{3/2}}{\sqrt{3}}|\tilde{\psi}|_{L^2(\omega)}.
$$

  On the other hand, since $\overline{u}$ is   an interior point of  $\mathcal{U}_{ad}$,
  $$
  |\eta|_{L^2(\omega\times(0, T))} =\beta |\overline{u}|_{L^2(\omega\times(0, T))} <\beta\rho_1<\beta.
  $$
  Hence,
  $$
  \displaystyle\frac{\beta}{\alpha}>\rho^*_2= \frac{M T^{3/2}}{\sqrt{3}}|\tilde{\psi}|_{L^2(\omega)}.
  $$
  When $\displaystyle\frac{\beta}{\alpha}\leq \rho^*_2$, there is a contradiction and therefore,
  $\overline{u}$ is   a boundary point of  $\mathcal{U}_{ad}$.

  \smallskip

  {\bf Case 2) } Assume that   $y_d(x, t)\leq -d_0$ in $Q$.
 It follows that $\displaystyle y_d(x, t)-\overline{y}(x, t)\leq - \frac{d_0}{2}
$ in $Q$.

 \smallskip

 In this case,  we can  construct a super-solution to  the equation (\ref{forwardc3}) and get the desired result
  similarly.

  Further, the above arguments still hold for the optimal control $\overline{v}$.
   \endpf

  \section{Proof of Theorem \ref{th2}}\label{chapter3}

  \noindent{\bf Proof of Theorem \ref{th2}. }  First,  since
  $\overline{u}$ and $\overline{v}$
are optimal controls, respectively,  for the optimal control problems ${\bf (P_1)}$
and  ${\bf (P_2)}$, we have that
$$
J_1(\overline{u})\leq J_1(0)\quad\mbox{ and }\quad J_2(\overline{v})\leq J_2(0).
$$
  Hence,  by  (\ref{App1}),
  \begin{eqnarray}\label{Add1*}
  \begin{array}{ll}
  &\displaystyle|\overline{u}|_{L^2(\omega\times(0, T))}+|\overline{v}|_{L^2(\omega\times(0, T))}
  +|\overline{y}-y_d|_{L^2(Q)}+|\overline{w}-y_d|_{L^2(Q)}\\[3mm]
  &\displaystyle\leq C\big(|y_d|_{L^2(Q)}+|y_0|_{W^{2, p}(\Omega)}\big).
  \end{array}
  \end{eqnarray}
  Further,  by the estimate of strong solutions and (\ref{Add1*}),
   \begin{eqnarray}\label{Add2*}
  \begin{array}{ll}
  &\displaystyle|\overline{y}|_{W^{2, 1}_2(Q)}+|\overline{w}|_{W^{2, 1}_2(Q)}
  \leq
  C\big(
  |\overline{u}|_{L^2(\omega\times(0, T))}+|\overline{v}|_{L^2(\omega\times(0, T))}+
  |y_0|_{H^{2}(\Omega)}
  \big)\\[3mm]
  &\displaystyle \leq C\big(|y_d|_{L^2(Q)}+|y_0|_{W^{2, p}(\Omega)}\big),
  \end{array}
  \end{eqnarray}
  and
  \begin{eqnarray}\label{Add3*}
  \begin{array}{ll}
  &\displaystyle|\eta|_{W^{2, 1}_2(Q)}+|z|_{W^{2, 1}_2(Q)}
  \leq
  C\big(|\overline{y}-y_d|_{L^2(Q)}+|\overline{w}-y_d|_{L^2(Q)}
  \big)\\[3mm]
  &\displaystyle \leq C\big(|y_d|_{L^2(Q)}+|y_0|_{W^{2, p}(\Omega)}\big).
  \end{array}
  \end{eqnarray}

  Since $|\overline{u}|_{L^p(\omega\times(0, T))}\leq \rho_1$ and
  $|\overline{v}|_{L^2(\omega\times(0, T))}\leq \rho_1$,
  we choose $r^*\in (n+2, p)$, and  by the interpolation inequality and (\ref{Add1*}), for a positive constant
  $\theta\in (0, 1)$,
  \begin{eqnarray*}
 &&|\overline{u}|_{L^{r^*}(\omega\times(0, T))}+|\overline{v}|_{L^{r^*}(\omega\times(0, T))}
 \leq |\overline{u}|^\theta_{L^{2}(\omega\times(0, T))}
 |\overline{u}|^{1-\theta}_{L^{p}(\omega\times(0, T))}+
 |\overline{v}|^\theta_{L^{2}(\omega\times(0, T))}
 |\overline{v}|^{1-\theta}_{L^{p}(\omega\times(0, T))}\\[1mm]
 &&\leq C\rho_1^{1-\theta} \big(|y_d|_{L^2(Q)}+|y_0|_{W^{2, p}(\Omega)}\big)^\theta\leq
 C\big(|y_d|_{L^2(Q)}+|y_0|_{W^{2, p}(\Omega)}\big)^\theta.
  \end{eqnarray*}
  Since $\rho_1<1$, this implies that
  \begin{eqnarray*}
  \begin{array}{ll}
  &\displaystyle |\overline{y}|_{W^{2, 1}_{r^*}(Q)}+|\overline{w}|_{W^{2, 1}_{r^*}(Q)}
  +|\eta|_{W^{2, 1}_{r^*}(Q)}+|z|_{W^{2, 1}_{r^*}(Q)}\\[2mm]
  &\displaystyle \leq
  C \big(|y_d|_{L^2(Q)}+|y_0|_{W^{2, p}(\Omega)}\big)^\theta+
  C|y_0|_{W^{2, r^*}(\Omega)}+
  C \big(|y_d|_{L^p(Q)}+|y_0|_{W^{2, p}(\Omega)}\big)^\theta\\[2mm]
  &\displaystyle \leq C \big(|y_d|_{L^p(Q)}+|y_0|_{W^{2, p}(\Omega)}\big)^\theta.
  \end{array}
  \end{eqnarray*}
  Therefore,
  \begin{eqnarray}\label{Add4*}
  \begin{array}{ll}
  &\displaystyle |\overline{y}|_{C^{1, 0}(\overline{Q})}+
  |\overline{w}|_{C^{1, 0}(\overline{Q})}+
  |\eta|_{C^{1, 0}(\overline{Q})}+|z|_{C^{1, 0}(\overline{Q})}\displaystyle
  \leq C \big(|y_d|_{L^p(Q)}+|y_0|_{W^{2, p}(\Omega)}\big)^\theta.
  \end{array}
  \end{eqnarray}

  Set $Y=\overline{y}-\overline{w}$ and $W=\eta-z$. Then,  $(Y, W)$ satisfies
  \begin{eqnarray}\label{c11****??}
	\left\{
	\begin{array}{ll}
		Y_t-\Delta Y=F_1
		+\chi_{\omega}(\overline{u}-\overline{v})&\mbox{ in }Q, \\ \ns
	W_t+\Delta W=F_2+
	\alpha Y			&\mbox{ in }Q, \\ \ns	
		Y=W=0 &\mbox{ on }\Sigma, \\ \ns
		Y(x, 0)=0,\  W(x, T)=0 &\mbox{ in }\Omega,
	\end{array}
	\right.
\end{eqnarray}
  where
  $$
  F_1=\sum\limits_{i,j=1}^{n}(\widetilde{a}^{i j}(\overline{y})\overline{y}_{x_i})_{x_j}
		-f(\overline{y}),$$
  and
  $$
  F_2=-\sum\limits_{i, j=1}^{n}\big(\widetilde{a}^{ij}(\overline{y})\eta_{x_i}\big)_{x_j}+
		\sum\limits_{i, j=1}^{n} (a^{i j})'(\overline{y})\overline{y}_{x_i}\eta_{x_j}
		+f'(\overline{y})\eta.
  $$
 By  (\ref{Add2*}),   (\ref{Add3*}) and  (\ref{Add4*}),  it is easy to show that
  \begin{eqnarray}\label{Add5*}
  \begin{array}{rl}
 &|F_1|_{L^2(Q)}+ |F_2|_{L^2(Q)}\\[2mm]
 &\leq C|\overline{y}|_{C^{1, 0}(\overline{Q})}|\overline{y}|_{W^{2, 1}_2(Q)}
 +C|\overline{y}|_{C^{1, 0}(\overline{Q})}|\eta|_{W^{2, 1}_2(Q)}
 \leq C\big(|y_d|_{L^p(Q)}+|y_0|_{W^{2, p}(\Omega)}\big)^{1+\theta}.
  \end{array}\end{eqnarray}

  On the other hand,  choose
  $u=\overline{v}$ in (\ref{c5}) and $u=\overline{u}$ in (\ref{PP1}). We obtain that
  \begin{equation}\label{Add6*}
  \beta |\overline{u}-\overline{v}|^2_{L^2(\omega\times(0, T))}
  \leq\int^T_0\int_{\omega}(\eta-z)(\overline{u}-\overline{v})dxdt.
  \end{equation}
 Multiplying the first equation of  (\ref{c11****??})
  by $W$
 and the second equation by $Y$,  integrating on $Q$, and using $Y(\cdot, 0)=
 W(\cdot, T)=0$,  we have that
 for any $\varepsilon>0$,
 \begin{eqnarray*}
 &&\alpha|Y|^2_{L^2(Q)}+\int^T_0\int_{\omega}(\eta-z)(\overline{u}-\overline{v})dxdt
 \leq \int_Q |F_1 W|dxdt+\int_Q |F_2 Y|dxdt\\[2mm]
 &&\leq \varepsilon |Y|^2_{L^2(Q)}
+\varepsilon |W|^2_{L^2(Q)}
+C|F_1|^2_{L^2(Q)}+ C|F_2|^2_{L^2(Q)}\\[2mm]
&&\leq \varepsilon |Y|^2_{L^2(Q)}
+\varepsilon C(|F_2|^2_{L^2(Q)}+\alpha|Y|^2_{L^2(Q)})
+C|F_1|^2_{L^2(Q)}+ C|F_2|^2_{L^2(Q)}.
 \end{eqnarray*}
 When $\varepsilon$ is sufficiently small,  combining the above estimate with (\ref{Add6*})
 and (\ref{Add5*}), we get that
 \begin{eqnarray*}
 &&|Y|_{L^2(Q)}+|\overline{u}-\overline{v}|_{L^2(\omega\times(0, T))}
 \leq C|F_1|_{L^2(Q)}+ C|F_2|_{L^2(Q)}\leq
 C\big(|y_d|_{L^p(Q)}+|y_0|_{W^{2, p}(\Omega)}\big)^{1+\theta}.
 \end{eqnarray*}
 Then,  by the estimate of strong solutions to (\ref{c11****??}),
 \begin{eqnarray*}
 &&|\overline{y}-\overline{w}|_{W^{2, 1}_2(Q)}=|Y|_{W^{2, 1}_2(Q)}\leq
 C|F_1|_{L^2(Q)}+ C|\overline{u}-\overline{v}|_{L^2(\omega\times(0, T))}\\[2mm]
 &&\leq C\big(|y_d|_{L^p(Q)}+|y_0|_{W^{2, p}(\Omega)}\big)^{1+\theta}.
 \end{eqnarray*}
 This finishes the proof of Theorem \ref{th2}.
 \endpf

  \section{Proofs of approximation results with  order $2$}\label{chapter4}

  First, we give a proof of Theorem \ref{th3}.

  \medskip

  \noindent{\bf Proof of Theorem \ref{th3}.} Assume that $\overline{u}$ and  $\overline{v}$
are interior points of $\mathcal{U}_{ad}$.
First, by  (\ref{Add}), we have the  optimality system for the optimal control problem ${\bf(P_1)}$:
\begin{eqnarray}\label{c9}
	\left\{
	\begin{array}{ll}
		\overline{y}_t-\sum\limits_{i,j=1}^{n}(a^{i j}(\overline{y})\overline{y}_{x_i})_{x_j}+f(\overline{y})=\chi_{\omega}\overline{u}=\chi_{\omega} \beta^{-1} \eta &\mbox{ in }Q, \\ \ns
	\eta_t+\sum\limits_{i, j=1}^{n}\big(a^{ij}(\overline{y})\eta_{x_i}\big)_{x_j}-
		\sum\limits_{i, j=1}^{n} (a^{i j})'(\overline{y})\overline{y}_{x_i}\eta_{x_j}-f'(\overline{y})\eta=\alpha (\overline{y}-y_d)		
		&\mbox{ in }Q, \\ \ns	
		\overline{y}=\eta=0 &\mbox{ on }\Sigma, \\ \ns
		\overline{y}(x, 0)=y_0(x),\  \eta(x, T)=0 &\mbox{ in }\Omega,
	\end{array}
	\right.
\end{eqnarray}
and the  optimality system for the optimal control problem ${\bf(P_2)}$:
\begin{eqnarray}\label{c10}
	\left\{
	\begin{array}{ll}
		\overline{w}_t-\Delta \overline{w}=\chi_{\omega}\overline{v}=\chi_{\omega} \beta^{-1} z &\mbox{ in }Q, \\ \ns
	z_t+\Delta z=\alpha (\overline{w}-y_d)	
		&\mbox{ in }Q, \\ \ns	
		\overline{w}=z=0 &\mbox{ on }\Sigma, \\ \ns
		\overline{w}(x, 0)=y_0(x),\  z(x, T)=0 &\mbox{ in }\Omega.
	\end{array}
	\right.
\end{eqnarray}

\medskip

Next, we give the error estimates on the optimal trajectories for the optimal control problems ${\bf (P_1)}$ and ${\bf (P_2)}$. Set $Y=\overline{y}-\overline{w}$ and $W=\eta-z$. By
(\ref{c9}) and (\ref{c10}),  $(Y, W)$ satisfies the following system:
\begin{eqnarray}\label{c11****}
	\left\{
	\begin{array}{ll}
		Y_t-\Delta Y=
		\sum\limits_{i,j=1}^{n}(\widetilde{a}^{i j}(\overline{y})\overline{y}_{x_i})_{x_j}-f(\overline{y})+\chi_{\omega} \beta^{-1} W&\mbox{ in }Q, \\ \ns
	W_t+\Delta W=
	-\sum\limits_{i, j=1}^{n}\big(\widetilde{a}^{ij}(\overline{y})\eta_{x_i}\big)_{x_j}+
		\sum\limits_{i, j=1}^{n} (a^{i j})'(\overline{y})\overline{y}_{x_i}\eta_{x_j}+f'(\overline{y})\eta+\alpha Y			&\mbox{ in }Q, \\ \ns	
		Y=W=0 &\mbox{ on }\Sigma, \\ \ns
		Y(x, 0)=0,\  W(x, T)=0 &\mbox{ in }\Omega.
	\end{array}
	\right.
\end{eqnarray}

In the following,  we estimate every  non-homogenous term
 in the right sides of the first and the second equations in (\ref{c11****}).
Since $\widetilde{a}^{i j}(0)=0$,  it follows that
\begin{eqnarray*}
&&\Big|\sum\limits_{i,j=1}^{n}(\widetilde{a}^{i j}(\overline{y})\overline{y}_{x_i})_{x_j}\Big|_{L^p(Q)}=\Big|\sum\limits_{i,j=1}^{n}
\Big[
\widetilde{a}^{i j}(\overline{y})\overline{y}_{x_i x_j}+
(\widetilde{a}^{i j})'(\overline{y})\overline{y}_{x_i}\overline{y}_{x_j}\Big]
\Big|_{L^p(Q)}\\
&&\leq C|\overline{y}|_{L^\infty(Q)}|\overline{y}|_{W^{2, 1}_p(Q)}
+C|\overline{y}|_{C^{1,0}(\overline{Q})}|\overline{y}|_{W^{2,1}_p(Q)}\leq C|\overline{y}|^2_{W^{2,1}_p(Q)}.
\end{eqnarray*}
Similarly,
\begin{eqnarray*}
&&\Big|\sum\limits_{i,j=1}^{n}(\widetilde{a}^{i j}(\overline{y})\eta_{x_i})_{x_j}\Big|_{L^p(Q)}=\Big|\sum\limits_{i,j=1}^{n}
\Big[
\widetilde{a}^{i j}(\overline{y})\eta_{x_i x_j}+
(\widetilde{a}^{i j})'(\overline{y})\eta_{x_i}\overline{y}_{x_j}\Big]
\Big|_{L^p(Q)}\\
&&\leq C|\overline{y}|_{L^\infty(Q)}|\eta|_{W^{2, 1}_p(Q)}
+C|\eta|_{C^{1,0}(\overline{Q})}|\nabla\overline{y}|_{L^p(Q)}\leq C|\eta|_{W^{2,1}_p(Q)}|\overline{y}|_{W^{2,1}_p(Q)},
\end{eqnarray*}
and
\begin{eqnarray*}
&&\Big|\sum\limits_{i,j=1}^{n}(\widetilde{a}^{i j})'
(\overline{y})\overline{y}_{x_i}\eta_{x_j}\Big|_{L^p(Q)}
\leq C|\eta|_{C^{1,0}(\overline{Q})} |\nabla\overline{y}|_{L^p(Q)}
\leq C|\eta|_{W^{2,1}_p(Q)} |\overline{y}|_{W^{2,1}_p(Q)}.
\end{eqnarray*}
Further,  since $f(0)=f'(0)=0$,  we have  that
$$
|f(\overline{y})|_{L^p(Q)}\leq C|f(\overline{y})|_{L^\infty(Q)}\leq C|\overline{y}|^2_{L^\infty(Q)}\leq C|\overline{y}|^2_{W^{2, 1}_p(Q)},
$$
and
$$
|f'(\overline{y})\eta|_{L^p(Q)}\leq |\overline{y}|_{L^\infty(Q)} |\eta|_{L^p(Q)}
\leq C|\eta|_{W^{2,1}_p(Q)}|\overline{y}|_{W^{2,1}_p(Q)}.
$$

\medskip

Furthermore,  set
\begin{eqnarray*}
&&F_1=\sum\limits_{i,j=1}^{n}(\widetilde{a}^{i j}(\overline{y})\overline{y}_{x_i})_{x_j}-f(\overline{y})
 \quad\mbox{and}\quad
F_2=-\sum\limits_{i, j=1}^{n}\big(\widetilde{a}^{ij}(\overline{y})\eta_{x_i}\big)_{x_j}+
		\sum\limits_{i, j=1}^{n} (a^{i j})'(\overline{y})\overline{y}_{x_i}\eta_{x_j}+f'(\overline{y})\eta.
\end{eqnarray*}
Then,   (\ref{c11****}) becomes the following system:
\begin{eqnarray}\label{c11*}
	\left\{
	\begin{array}{ll}
		Y_t-\Delta Y=
		F_1+\chi_{\omega} \beta^{-1} W&\mbox{ in }Q, \\ \ns
	W_t+\Delta W=
	F_2+\alpha Y
				&\mbox{ in }Q, \\ \ns	
		Y=W=0 &\mbox{ on }\Sigma, \\ \ns
		Y(x, 0)=0,\  W(x, T)=0 &\mbox{ in }\Omega,
	\end{array}
	\right.
\end{eqnarray}
where
\begin{equation}\label{q1}
|F_1|_{L^p(Q)}\leq C|\overline{y}|^2_{W^{2,1}_p(Q)}\quad\mbox{ and }\quad|F_2|_{L^p(Q)}\leq C|\eta|_{W^{2,1}_p(Q)}|\overline{y}|_{W^{2,1}_p(Q)}.
\end{equation}

\medskip

Multiplying the first and second equations in (\ref{c11*}), respectively,  by
$W$ and $Y$,  we get that for any $\varepsilon>0$,
\begin{eqnarray*}
&&\int^T_0\int_{\omega} \beta^{-1} W^2 dxdt+\int_Q \alpha Y^2 dxdt
=-\int_Q (F_1 W+F_2Y)dxdt\\[2mm]
&&\leq \varepsilon\int_Q W^2dxdt+C\int_Q F_1^2dxdt
+\varepsilon\int_Q Y^2dxdt+C\int_Q F^2_2dxdt.
\end{eqnarray*}
When $\varepsilon$ is sufficiently small, it follows that
$$
|W|_{L^2(\omega\times(0, T))}+|Y|_{L^2(Q)}
\leq C\big(|F_1|_{L^2(Q)}+|F_2|_{L^2(Q)}+\sqrt{\varepsilon}|W|_{L^2(Q)}\big).
$$
By the well-posedness results for strong solutions to (\ref{c11*}) and the above estimate,
\begin{eqnarray*}
&&|W|_{W^{2,1}_2(Q)}+|Y|_{W^{2,1}_2(Q)}
\leq C\big(|F_1|_{L^2(Q)}+|F_2|_{L^2(Q)}\big)+
C\big(|Y|_{L^2(Q)}+|W|_{L^2(\omega\times(0, T))}\big)\\[2mm]
&&\leq C\big(|F_1|_{L^2(Q)}+|F_2|_{L^2(Q)}+\sqrt{\varepsilon}|W|_{L^2(Q)}\big).
\end{eqnarray*}
When $\varepsilon$ is sufficiently small, it follows that
\begin{eqnarray*}
&&|W|_{W^{2,1}_2(Q)}+|Y|_{W^{2,1}_2(Q)}
\leq C\big(|F_1|_{L^2(Q)}+|F_2|_{L^2(Q)}\big).
\end{eqnarray*}

\medskip

By the Sobolev embedding theorem,
if $n\leq 2$,  then for any $s\geq 2$,
$$|W|_{L^s(Q)}+|Y|_{L^s(Q)}\leq C\big(|W|_{W^{2,1}_2(Q)}+|Y|_{W^{2,1}_2(Q)}\big);$$
Otherwise, for $p_1=\frac{2(n+2)}{n-2}$,
$$|W|_{L^{p_1}(Q)}+|Y|_{L^{p_1}(Q)}\leq C\big(|W|_{W^{2,1}_2(Q)}+|Y|_{W^{2,1}_2(Q)}\big)
\leq C\big(|F_1|_{L^2(Q)}+|F_2|_{L^2(Q)}\big).$$

\medskip

Further, using the well-posedness results for strong solutions to (\ref{c11*}) again,  one gets that
 \begin{eqnarray*}
&&|W|_{W^{2,1}_{p_1}(Q)}+|Y|_{W^{2,1}_{p_1}(Q)}
\leq C\big(|F_1|_{L^{p_1}(Q)}+|F_2|_{L^{p_1}(Q)}\big)+
C\big(|Y|_{L^{p_1}(Q)}+|W|_{L^{p_1}(\omega\times(0, T))}\big)\\[2mm]
&&\leq C\big(|F_1|_{L^{p_1}(Q)}+|F_2|_{L^{p_1}(Q)}\big)+C\big(|F_1|_{L^{2}(Q)}+|F_2|_{L^{2}(Q)}\big)\leq C\big(|F_1|_{L^{p_1}(Q)}+|F_2|_{L^{p_1}(Q)}\big).
\end{eqnarray*}
By the Sobolev embedding theorem again,
if $n+2\leq 2 p_1$,  then for any $s\geq p_1$,
$$|W|_{L^s(Q)}+|Y|_{L^s(Q)}\leq C\big(|W|_{W^{2,1}_{p_1}(Q)}+|Y|_{W^{2,1}_{p_1}(Q)}\big);$$
Otherwise, for $p_2=\frac{p_1(n+2)}{n+2-2p_1}$,
$$|W|_{L^{p_2}(Q)}+|Y|_{L^{p_2}(Q)}\leq C\big(|W|_{W^{2,1}_{p_1}(Q)}+|Y|_{W^{2,1}_{p_1}(Q)}\big)
\leq C\big(|F_1|_{L^{p_1}(Q)}+|F_2|_{L^{p_1}(Q)}\big).$$

\medskip

If we repeat the above arguments,  there always  exists an $N\in \mathbb{N}$, such that
$p_N\leq p < p_{N+1}$. At this moment,
$$|W|_{L^{p_{N+1}}(Q)}+|Y|_{L^{p_{N+1}}(Q)}\leq C\big(|W|_{W^{2,1}_{p_N}(Q)}+|Y|_{W^{2,1}_{p_N}(Q)}\big)
\leq C\big(|F_1|_{L^{p_N}(Q)}+|F_2|_{L^{p_N}(Q)}\big).$$
Therefore, using the $L^p$-estimates of linear parabolic equations,
by (\ref{q1}), we have that
\begin{eqnarray}\label{qq8}
\begin{array}{rl}
&|W|_{W^{2, 1}_p(Q)}+|Y|_{W^{2, 1}_p(Q)}
\leq C\big(|F_1|_{L^p(Q)}+|F_2|_{L^p(Q)}\big)
+C
\big(|W|_{L^{p_{N+1}}(Q)}+|Y|_{L^{p_{N+1}}(Q)}\big)\\[3mm]
&\leq C\big(|F_1|_{L^p(Q)}+|F_2|_{L^p(Q)}\big)
+C\big(|F_1|_{L^{p_N}(Q)}+|F_2|_{L^{p_N}(Q)}\big)\\[3mm]
&\leq C\big(|F_1|_{L^p(Q)}+|F_2|_{L^p(Q)}\big)\\[3mm]
&\leq C|\overline{y}|^2_{W^{2,1}_p(Q)}+
C|\eta|_{W^{2,1}_p(Q)}|\overline{y}|_{W^{2,1}_p(Q)}.
\end{array}
\end{eqnarray}

Furthermore, since $\overline{u}$ is the optimal control of the problem ${\bf(P_1)}$,
$$J_1(\overline{u})\leq J_1(0)\leq C\big(|y(\cdot, \cdot; 0)|^2_{L^2(Q)}+
|y_d|^2_{L^2(Q)}\big)\leq C\big(|y_0|^2_{H^{2}(\Omega)}+
|y_d|^2_{L^2(Q)}\big),$$
where $y(\cdot, \cdot; 0)$ denotes the solution to (\ref{aeq11})
 associated to $u\equiv 0$ and the initial value $y_0$.
This implies that
$$
|\overline{y}-y_d|_{L^2(Q)}+|\overline{u}|_{L^2(\omega\times(0, T))}\leq
C\big(|y_0|_{H^{2}(\Omega)}+
|y_d|_{L^2(Q)}\big).
$$
Hence,  by the well-posedness results for (\ref{c9}), we have
that
$$
|\overline{y}|_{W^{2, 1}_2(Q)}\leq C\big(|\overline{u}|_{L^2(\omega\times(0, T))}
+|y_0|_{H^{2}(\Omega)}\big)
\leq C\big(|y_0|_{H^{2}(\Omega)}+
|y_d|_{L^2(Q)}\big),
$$
and
$$
|\eta|_{W^{2, 1}_2(Q)}\leq
C|\overline{y}-y_d|_{L^2(Q)}
\leq C\big(|y_0|_{H^{2}(\Omega)}+
|y_d|_{L^2(Q)}\big).
$$

\medskip

By the Sobolev embedding theorem,
if $n\leq 2$,  then for any $s\geq 2$,
$$|\overline{y}|_{L^s(Q)}+|\eta|_{L^s(Q)}\leq C\big(|\overline{y}|_{W^{2,1}_2(Q)}
+|\eta|_{W^{2,1}_2(Q)}\big);$$
Otherwise, for $p_1=\frac{2(n+2)}{n-2}$,
$$|\overline{y}|_{L^{p_1}(Q)}+|\eta|_{L^{p_1}(Q)}\leq C\big(|\overline{y}|
_{W^{2,1}_2(Q)}+|\eta|_{W^{2,1}_2(Q)}\big)
\leq C\big(|y_0|_{H^{2}(\Omega)}+
|y_d|_{L^2(Q)}\big).$$

\medskip

Further, by the well-posedness results for (\ref{c9}) again, we have
that
$$
|\overline{y}|_{W^{2, 1}_{p_1}(Q)}\leq C\big(|\eta|_{L^{p_1}(\omega\times(0, T))}
+|y_0|_{W^{2, p_1}(\Omega)}\big)
\leq C\big(|y_0|_{W^{2, p_1}(\Omega)}+
|y_d|_{L^2(Q)}\big),
$$
and
$$
|\eta|_{W^{2, 1}_{p_1}(Q)}\leq
C|\overline{y}-y_d|_{L^{p_1}(Q)}
\leq C\big(|y_0|_{W^{2, p_1}(\Omega)}+
|y_d|_{L^{p_1}(Q)}\big).
$$
By the Sobolev embedding theorem again,
if $n+2\leq 2 p_1$,  then for any $s\geq p_1$,
$$|\overline{y}|_{L^s(Q)}+|\eta|_{L^s(Q)}\leq
C\big(|\overline{y}|_{W^{2,1}_{p_1}(Q)}+|\eta|_{W^{2,1}_{p_1}(Q)}\big);$$
Otherwise, for $p_2=\frac{p_1(n+2)}{n+2-2p_1}$,
$$|\overline{y}|_{L^{p_2}(Q)}+|\eta|_{L^{p_2}(Q)}
\leq C\big(|\overline{y}|_{W^{2,1}_{p_1}(Q)}+|\eta|_{W^{2,1}_{p_1}(Q)}\big)
\leq C\big(|y_0|_{W^{2, p_1}(\Omega)}+
|y_d|_{L^{p_1}(Q)}\big).$$
Repeating the above arguments, we can obtain that
$$
|\overline{y}|_{W^{2,1}_{p}(Q)}+|\eta|_{W^{2,1}_{p}(Q)}
\leq C\big(|y_0|_{W^{2, p}(\Omega)}+
|y_d|_{L^p(Q)}\big).
$$
The above estimate, together with (\ref{qq8}), shows that
$$
|W|_{W^{2, 1}_p(Q)}+|Y|_{W^{2, 1}_p(Q)}
\leq C\big(|y_0|^2_{W^{2, p}(\Omega)}+
|y_d|^2_{L^p(Q)}\big).
$$
Since $\overline{u}-\overline{v}=\beta^{-1}W$ in $\omega\times(0, T)$
in the interior-point case, the same estimate gives the
asserted bound for $|\overline{u}-\overline{v}|_{L^p(\omega\times(0, T))}$. This finishes the proof of Theorem \ref{th3}.

\endpf

\medskip

Next,  we give a proof of  Theorem \ref{th4}.

\medskip

\noindent {\bf Proof of Theorem \ref{th4}. }  The whole proof is divided into three parts.

\smallskip

{\bf Step 1. } First,  we give the necessary conditions satisfied by the optimal controls
$u^*$ and $v^*$.

\medskip

For
the optimal control $u^*\in \mathcal{U}_{ad}$ of the problem ${\bf (\widehat{P}_1)}$,
 any $u\in \mathcal{U}_{ad}$ and $\varepsilon\in(0, 1)$, set
$u_\varepsilon=u^*+\varepsilon(u-u^*)$. Then
$u_\varepsilon\in \mathcal{U}_{ad}$ and
\begin{equation}\label{c1LWZ}
\lim\limits_{\varepsilon\rightarrow 0}
\frac{\widehat{J}_1(u_\varepsilon)-\widehat{J}_1(u^*)}{\varepsilon}=
\alpha\int_Q |y^*-y_d|^{\hat{p}-2}(y^*-y_d)\xi dxdt+
\beta\int_0^T\int_{\omega} u^*(u-u^*) dxdt\geq 0,
\end{equation}
where $\xi\in L^2(0, T; H^1_0(\Omega))\cap
C([0, T]; L^2(\Omega))$ satisfies the following linear parabolic equation:
\begin{eqnarray}\label{c2LWZ}
	\left\{
	\begin{array}{ll}
		\xi_t-\sum\limits_{i, j=1}^{n}\big(a^{ij}(y^*)\xi_{x_i}\big)_{x_j}-
		\sum\limits_{i, j=1}^{n}\big[(a^{i j})'(y^*)y^*_{x_i}\xi\big]_{x_j}+f'(y^*)\xi=\chi_{\omega}
		(u-u^*)	
		&\mbox{ in }Q, \\ \ns
		\xi=0 &\mbox{ on }\Sigma, \\ \ns
		\xi(x, 0)=0 &\mbox{ in }\Omega. \\ \ns
	\end{array}
	\right.
\end{eqnarray}

Next, introduce the following linear parabolic equation:
\begin{eqnarray}\label{c3LWZ}
	\left\{
	\begin{array}{ll}
		\eta_t+\sum\limits_{i, j=1}^{n}\big(a^{ij}(y^*)\eta_{x_i}\big)_{x_j}-
		\sum\limits_{i, j=1}^{n} (a^{i j})'(y^*)y^*_{x_i}\eta_{x_j}-f'(y^*)\eta
		=\alpha|y^*-y_d|^{\hat{p}-2}(y^*-y_d)
		&\mbox{ in }Q, \\ \ns
		\eta=0 &\mbox{ on }\Sigma, \\ \ns
		\eta(x, T)=0 &\mbox{ in }\Omega. \\ \ns
	\end{array}
	\right.
\end{eqnarray}
In (\ref{c3LWZ}),  since $y^*\in W^{2, 1}_p(Q)$ and $p>n+2$,
we have that for $i, j=1, \cdots, n$,
$$
y^*,  a^{i j}(y^*)\in C^{1, 0}(\overline{Q}),  \quad
 (a^{i j})'(y^*)y^*_{x_i}, f'(y^*)\in C(\overline{Q}),\quad
 \mbox{and} \quad  y^*-y_d
	 \in L^\infty(Q).
$$
Hence, (\ref{c3LWZ}) admits a unique  strong solution
$\eta\in W^{2,1}_p(Q)$.
Moreover,
\begin{equation}\label{new1LWZ}
|\eta|_{W^{2, 1}_p(Q)}\leq C\alpha |y^*-y_d|^{\hat{p}-1}_{L^\infty(Q)}.
\end{equation}
Multiplying both sides of  the first equation in (\ref{c2LWZ}) by $\eta$ and integrating it
on $Q$, we obtain that
\begin{equation}\label{c4LWZ}
\alpha\int_Q |y^*-y_d|^{\hat{p}-2}(y^*-y_d) \xi dxdt+
\int_0^T\int_{\omega} (u-u^*)\eta dxdt=0.
\end{equation}
This, together with (\ref{c1LWZ}), indicates that
\begin{equation}\label{c5LWZ}
\int^T_0\int_\omega (\beta u^*-\eta)(u-u^*)
dxdt\geq 0,
\quad \forall \ u\in \mathcal{U}_{ad}.
\end{equation}

On the other hand, using the similar arguments above  for the optimal control problem ${\bf (\widehat{P}_2)}$ and recalling that $(v^*, w^*)$ is the unique optimal pair,  we get that
\begin{equation}\label{PP1LWZ}
\displaystyle\int^T_0\int_\omega
(\beta v^*-z)(u-v^*)dxdt\geq 0,
\quad \forall \ u\in \mathcal{U}_{ad},
\end{equation}
where $z$ satisfies the linear parabolic equation:
\begin{eqnarray}\label{c7LWZ}
	\left\{
	\begin{array}{ll}
		z_t+\Delta z=\alpha|w^*-y_d|^{\hat{p}-2}(w^*-y_d)
		&\mbox{ in }Q, \\ \ns
		z=0 &\mbox{ on }\Sigma, \\ \ns
		z(x, T)=0 &\mbox{ in }\Omega. \\ \ns
	\end{array}
	\right.
\end{eqnarray}
Similar to  (\ref{new1LWZ}),  we also have that
\begin{equation}\label{new1*LWZ}
|z|_{W^{2, 1}_p(Q)}\leq C\alpha |w^*-y_d|^{\hat{p}-1}_{L^{\infty}(Q)}.
\end{equation}

\medskip

\medskip

{\bf Step 2. } In this step,  we give the expressions of the optimal controls
$u^*$ and $v^*$.

\medskip

{\bf Case 1. }  Assume that $|u^*|_{L^p(\omega\times(0, T))}<\rho_1$
and $|v^*|_{L^p(\omega\times(0, T))}<\rho_1$, i.e.,
$u^*$ and $v^*$ are two  interior points of $\mathcal{U}_{ad}$.

\medskip

In this case,  by (\ref{c5LWZ}) and (\ref{PP1LWZ}), it follows that
\begin{equation}\label{AddLWZ}
\displaystyle u^*=\chi_\omega \beta^{-1} \eta\ \ \quad
\mbox{ and } \ \  \quad v^*=\chi_\omega \beta^{-1} z
\quad\quad\mbox{a.e.}\quad\mbox{ in }\ Q.
\end{equation}

\medskip

{\bf Case 2. }  Assume that $|u^*|_{L^p(\omega\times(0, T))}=\rho_1$
or $|v^*|_{L^p(\omega\times(0, T))}=\rho_1$, i.e.,
$u^*$ or $v^*$ is a  boundary point of $\mathcal{U}_{ad}$.

\medskip

If $u^*$  is a  boundary point of $\mathcal{U}_{ad}$, then  (\ref{lwz4}) and (\ref{lwz8}) hold for $u^*$.
Meanwhile,  by (\ref{new1LWZ}) and (\ref{App1}),  it follows that
\begin{eqnarray*}
&&|\eta|_{W^{2, 1}_p(Q)}\leq C\alpha|y^*-y_d|^{\hat{p}-1}_{L^\infty(Q)}
\leq C\alpha \big(|u^*|_{L^p(\omega\times(0, T))}+
|y_0|_{W^{2, p}(\Omega)}+
|y_d|_{L^\infty(Q)}\big)^{\hat{p}-1}
\leq C\alpha \rho_1^{\hat{p}-1}.
\end{eqnarray*}
Hence,
\begin{equation}\label{LWZ001}
\beta\rho_1\leq |\eta|_{L^p(\omega\times(0, T))}
\leq |\eta|_{W^{2,1}_p(Q)}\leq C\alpha\rho_1^{\hat{p}-1}.\end{equation}
Since $\hat{p}>2$,  we have that (\ref{LWZ001}) fails for a sufficiently small $\rho_1$.
This indicates that there exists a positive
constant $\rho_4$,
such that {\bf Case 2} will not
happen for $u^*$ when $|y_0|_{W^{2, p}(\Omega)} + |y_d|_{L^\infty(Q)}\leq \rho_4$ and
$\rho_1$ in $\mathcal{U}_{ad}$  is replaced by $\rho_4$.
The same argument,  applied to $v^*$ and $z$,  excludes the boundary
case for $v^*$ under the same smallness condition.

\medskip


\medskip

Hence, in such a case,  we have the  optimality system for the optimal control problem ${\bf(\widehat{P}_1)}$:
\begin{eqnarray}\label{c9LWZ}
	\left\{
	\begin{array}{ll}
		y^*_t-\sum\limits_{i,j=1}^{n}(a^{i j}(y^*)y^*_{x_i})_{x_j}+f(y^*)=\chi_{\omega}u^*=\chi_{\omega} \beta^{-1} \eta &\mbox{ in }Q, \\ \ns
	\eta_t+\sum\limits_{i, j=1}^{n}\big(a^{ij}(y^*)\eta_{x_i}\big)_{x_j}-
		\sum\limits_{i, j=1}^{n} (a^{i j})'(y^*)y^*_{x_i}\eta_{x_j}-f'(y^*)\eta
		=\alpha|y^*-y_d|^{\hat{p}-2}(y^*-y_d)		
		&\mbox{ in }Q, \\ \ns	
		y^*=\eta=0 &\mbox{ on }\Sigma, \\ \ns
		y^*(x, 0)=y_0(x),\  \eta(x, T)=0 &\mbox{ in }\Omega
	\end{array}
	\right.
\end{eqnarray}
and the  optimality system for the optimal control problem ${\bf(\widehat{P}_2)}$:
\begin{eqnarray}\label{c10LWZ}
	\left\{
	\begin{array}{ll}
		\displaystyle w^*_t-\Delta w^*=\chi_{\omega}v^*=
		\chi_\omega \beta^{-1} z  &\mbox{ in }Q, \\ \ns
	z_t+\Delta z=\alpha|w^*-y_d|^{\hat{p}-2}(w^*-y_d)		
		&\mbox{ in }Q, \\ \ns	
		w^*=z=0 &\mbox{ on }\Sigma, \\ \ns
		w^*(x, 0)=y_0(x),\  z(x, T)=0 &\mbox{ in }\Omega.
	\end{array}
	\right.
\end{eqnarray}

\medskip

{\bf Step 3. } We give the error estimates on the optimal trajectories.

\medskip

First,  set $Y=y^*-w^*$ and $W=\eta-z$. By
(\ref{c9LWZ}) and (\ref{c10LWZ}),  $(Y, W)$ satisfy the following system:
\begin{eqnarray}\label{c11***}
	\left\{
	\begin{array}{ll}
		Y_t-\Delta Y=
		F_1+\chi_{\omega} \beta^{-1} W&\mbox{ in }Q, \\ \ns
	W_t+\Delta W=
	F_2+F_3		&\mbox{ in }Q, \\ \ns	
		Y=W=0 &\mbox{ on }\Sigma, \\ \ns
		Y(x, 0)=0,\  W(x, T)=0 &\mbox{ in }\Omega,
	\end{array}
	\right.
\end{eqnarray}
where
\begin{eqnarray*}
&&F_1=\sum\limits_{i,j=1}^{n}(\widetilde{a}^{i j}(y^*)y^*_{x_i})_{x_j}-f(y^*),
 \quad
F_2=-\sum\limits_{i, j=1}^{n}\big(\widetilde{a}^{ij}(y^*)\eta_{x_i}\big)_{x_j}+
		\sum\limits_{i, j=1}^{n} (a^{i j})'(y^*)y^*_{x_i}\eta_{x_j}+f'(y^*)\eta,\\
		&&F_3=\alpha|y^*-y_d|^{\hat{p}-2} (y^*-y_d)-\alpha|w^*-y_d|^{\hat{p}-2} (w^*-y_d).
\end{eqnarray*}

Similar to  (\ref{q1}),
$$
\displaystyle |F_1|_{L^p(Q)}\leq C|y^*|^2_{W^{2,1}_p(Q)}\quad\mbox{ and }\quad|F_2|_{L^p(Q)}\leq C|\eta|_{W^{2,1}_p(Q)}|y^*|_{W^{2,1}_p(Q)}.
$$
Moreover,
\begin{eqnarray*}
&&|F_3|\leq C(|y^*-y_d|^{\hat{p}-2}+|w^*-y_d|^{\hat{p}-2})|y^*-w^*|
\leq C(|y^*|^{\hat{p}-2}_{L^\infty(Q)}+|w^*|^{\hat{p}-2}_{L^\infty(Q)}+
|y_d|^{\hat{p}-2}_{L^\infty(Q)})|Y|\\[2mm]
&&\leq C(|y_0|^{\hat{p}-2}_{W^{2, p}(\Omega)}+|u^*|^{\hat{p}-2}_{L^p(\omega\times(0, T))}
+|v^*|^{\hat{p}-2}_{L^p(\omega\times(0, T))}+
|y_d|^{\hat{p}-2}_{L^\infty(Q)})|Y|
\leq C\rho_4^{\hat{p}-2}|Y|.
\end{eqnarray*}

Furthermore, using the $L^p$-estimates of linear parabolic equations for the first and second equations in
(\ref{c11***}),  respectively, we get that
$$
|Y|_{W^{2, 1}_p(Q)}\leq C\big(|F_1|_{L^p(Q)}+
|W|_{L^p(Q)}
\big)
$$
and
$$
|W|_{W^{2, 1}_p(Q)}\leq C\big(|F_2|_{L^p(Q)}+
\rho_4^{\hat{p}-2}|Y|_{L^p(Q)}
\big).
$$
When $\rho_4$ is sufficiently small,  it follows that
\begin{equation}\label{lwz10LWZ}
|Y|_{W^{2, 1}_p(Q)}+
|W|_{W^{2, 1}_p(Q)}\leq
 C\big(|F_1|_{L^p(Q)}+
|F_2|_{L^p(Q)}
\big)\leq C\big(
|y^*|^2_{W^{2,1}_p(Q)}+|\eta|_{W^{2,1}_p(Q)}|y^*|_{W^{2,1}_p(Q)}
\big).
\end{equation}

\medskip

Finally,  we give an estimate on the solution $(y^*, \eta)$ to
(\ref{c9LWZ}). Using the $L^p$-estimates of linear parabolic equations again for the first and second equations in
(\ref{c9LWZ}),  respectively, we get that
$$
|y^*|_{W^{2, 1}_p(Q)}\leq C\big(|\eta|_{L^p(Q)}+
|y_0|_{W^{2, p}(\Omega)}
\big),
$$
and
\begin{eqnarray*}
&&|\eta|_{W^{2, 1}_p(Q)}\leq C|y^*-y_d|^{\hat{p}-1}_{L^\infty(Q)}
\\[2mm]
&&
\leq
C\big[
|y_d|^{\hat{p}-1}_{L^\infty(Q)}+\big(|u^*|_{L^p(\omega\times(0, T))}+
|y_0|_{W^{2, p}(\Omega)}
\big)^{\hat{p}-2}|y^*|_{W^{2, 1}_p(Q)}\big]
\leq C\big(
|y_d|^{\hat{p}-1}_{L^\infty(Q)}+\rho_4^{\hat{p}-2} |y^*|_{W^{2, 1}_p(Q)}
\big).
\end{eqnarray*}
When $\rho_4$ is sufficiently small,  it follows that
\begin{equation}\label{lwz10*LZW}
|y^*|_{W^{2, 1}_p(Q)}+
|\eta|_{W^{2, 1}_p(Q)}\leq
 C\big(|y_0|_{W^{2, p}(\Omega)}+
|y_d|_{L^\infty(Q)}
\big).
\end{equation}
By (\ref{lwz10LWZ}) and (\ref{lwz10*LZW}), we obtain that
$$
|Y|_{W^{2, 1}_p(Q)}+
|W|_{W^{2, 1}_p(Q)}\leq
C\big(|y_0|^2_{W^{2, p}(\Omega)}+
|y_d|^2_{L^\infty(Q)}
\big).
$$
Since $u^*-v^*= \beta^{-1} W$  in  $\omega\times(0, T)$, the same bound also gives the asserted estimate for
$|u^*-v^*|_{L^p(\omega\times(0,T))}$.
This finishes the proof of Theorem \ref{th4}.  \endpf

\medskip

\section{Appendix}

This section is devoted to proving the existence of optimal controls and Proposition \ref{prop1}.

\medskip

First, we prove the existence of   optimal controls for the problem ${\bf(P_1)}$.
\begin{lemma}\label{Aplemma}
There exists at least an optimal control for the optimal control problem ${\bf(P_1)}$.
\end{lemma}

\noindent {\bf Proof. }  Since  $J_1(\cdot)\geq 0$,
there exists a minimizing sequence $\{u_k\}_{k=1}^\infty\subseteq \mathcal{U}_{ad}$,
such that
$$
\lim\limits_{k\rightarrow\infty}J_1(u_k(\cdot))=\inf\limits_{u(\cdot)\in \mathcal{U}_{ad}} J_1(u(\cdot)).
$$
 Since  $|u_k|_{L^p(\omega\times(0, T))}\leq \rho_1$ for any $k\in\mathbb{N}$,   there exists a subsequence of  $\{u_k\}_{k=1}^\infty$ (still denoted by itself) and
 $\overline{u}\in \mathcal{U}_{ad}$, such that
 $$
 u_k\rightarrow \overline{u}\quad\quad\mbox{ weakly in } \quad L^p(\omega\times(0, T)),\quad \mbox{ as } k\rightarrow \infty.
 $$
Denote by $y_k$ the solution to (\ref{aeq11}) associated to $u=u_k$.  By the local well-posedness results for (\ref{aeq11}), there exists a subsequence of  $\{y_k\}_{k=1}^\infty$ (still denoted by itself) and
 $\overline{y}\in W^{2, 1}_p(Q)$, such that
 \begin{eqnarray*}
 &&y_k\rightarrow \overline{y}\quad\quad\mbox{ weakly in } \quad W^{2, 1}_p(Q),\quad \mbox{ as } k\rightarrow \infty;\\[3mm]
 &&y_k\rightarrow \overline{y}\quad\quad\mbox{ strongly in } \quad C^{1, 0}(\overline{Q}),\quad \mbox{ as } k\rightarrow \infty.
 \end{eqnarray*}

Hence,  $\overline{y}$ is the solution to (\ref{aeq11}) associated to $u=\overline{u}$, and
$$
\varliminf\limits_{k\rightarrow\infty} J_1(u_k)\geq J_1(\overline{u}).
$$
This indicates that $J_1(\overline{u})=\inf\limits_{u(\cdot)\in \mathcal{U}_{ad}} J_1(u(\cdot))$, and $\overline{u}\in \mathcal{U}_{ad}$ is an optimal control
 of ${\bf(P_1)}$.   \endpf

\medskip

\medskip

Next, we give a proof of Proposition \ref{prop1}.

\medskip

\noindent {\bf Proof of Proposition \ref{prop1}. } First,  let $z=y-w$,  where
$y$ and $w$ are the corresponding solutions to
$(\ref{aeq11})$ and $(\ref{aeq1*})$ with $(v=u)$. Then, $z\in W^{2, 1}_p(Q)$
satisfies the following equation:
\begin{eqnarray}\label{appendix1}
	\left\{
	\begin{array}{ll}
		z_t-\Delta z=\sum\limits_{i,j=1}^{n}(\widetilde{a}^{i j}(y)y_{x_i})_{x_j}-f(y) &\mbox{ in }Q, \\ \ns
		z=0 &\mbox{ on }\Sigma, \\ \ns
		z(x, 0)=0 &\mbox{ in }\Omega. \\ \ns
	\end{array}
	\right.
\end{eqnarray}
By the $L^p$-estimates of linear parabolic equations for (\ref{appendix1}),  it follows  that
\begin{eqnarray}\label{appendix2}
\begin{array}{rl}
&\displaystyle |z|_{W^{2,1}_p(Q)}\leq C\big|\sum\limits_{i,j=1}^{n}(\widetilde{a}^{i j}(y)y_{x_i})_{x_j}\big|_{L^p(Q)}+C|f(y)|_{L^p(Q)}\\
&\displaystyle\leq C\sum\limits_{i,j=1}^{n}\big|\widetilde{a}^{i j}(y)y_{x_i x_j}\big|_{L^p(Q)}
+C\sum\limits_{i,j=1}^{n}\big|(\widetilde{a}^{i j})'(y)y_{x_i}y_{x_j}\big|_{L^p(Q)}
+C|f(y)|_{L^p(Q)}.
\end{array}
\end{eqnarray}

Notice that for any $i, j=1, \cdots, n$, $\widetilde{a}^{i j}(0)=f(0)=f'(0)=0$.
This implies that
$$|\widetilde{a}^{i j}(y)|_{L^\infty(Q)}
=|\widetilde{a}^{i j}(y)-\widetilde{a}^{i j}(0)|_{L^\infty(Q)}\leq C|y|_{L^\infty(Q)}
\leq C|y|_{W^{2,1}_p(Q)},
$$and
$$|f(y)|_{L^p(Q)}
\leq C|y^2|_{L^p(Q)}\leq C|y|_{L^\infty(Q)}|y|_{L^p(Q)}\leq C|y|^2_{W^{2, 1}_p(Q)}.
$$
Hence,  it holds that for any $i, j=1, \cdots, n$,
\begin{eqnarray*}
&&\big|\widetilde{a}^{i j}(y)y_{x_i x_j}\big|_{L^p(Q)}
\leq C|\widetilde{a}^{i j}(y)|_{L^\infty(Q)}|y|_{W^{2,1}_p(Q)}
\leq C|y|^2_{W^{2,1}_p(Q)},\\[2mm]
&&\mbox{and  } \big|(\widetilde{a}^{i j})'(y)y_{x_i}y_{x_j}\big|_{L^p(Q)}
\leq C|y|_{C^{1, 0}(\overline{Q})} |\nabla y|_{L^p(Q)}
\leq C|y|^2_{W^{2,1}_p(Q)}.
\end{eqnarray*}
Combining the above estimates with (\ref{appendix2}), we obtain that
$$\displaystyle |z|_{W^{2,1}_p(Q)}\leq C
|y|^2_{W^{2,1}_p(Q)}.$$
This, together with (\ref{App1}), implies that
$$
\displaystyle |y-w|_{W^{2,1}_p(Q)}=|z|_{W^{2,1}_p(Q)}
\leq C
|y|^2_{W^{2,1}_p(Q)}
\leq C
\big(|u|^2_{L^p(Q)}+|y_0|^2_{W^{2, p}(\Omega)}\big).
$$
This completes the proof of Proposition \ref{prop1}.  \endpf


\end{document}